\documentclass[reqno,a4paper,12pt]{amsart}
\usepackage{amsmath,amssymb,amsthm,amsfonts}
\usepackage[]{graphicx}
\usepackage{setspace} 
\usepackage[mathscr]{eucal}

\numberwithin{equation}{section}
{\theoremstyle{definition}
\newtheorem{defn}{Definition}[section]}
\newtheorem{theorem}{Theorem}[section]
\newtheorem{proposition}[theorem]{Proposition}

\newtheorem{corollary}[theorem]{Corollary}
\newtheorem{lemma}[theorem]{Lemma}

{\theoremstyle{definition}
{
\newtheorem{remark}[theorem]{Remark}

}}

\newtheorem*{theorem*}{Theorem}

\newcommand{\cal}{\mathcal}

\newcommand{\OO}{{\cal O}}

\newcommand{\Nn}{{\mathbb{N}}}

\newcommand{\Rr}{{\mathbb{R}}}
\newcommand{\Ss}{{\mathbb{S}}}

\newcommand{\Hscr}{\mathscr{H}}

\newcommand{\Pscr}{\mathscr{P}}

\newcommand{\Vscr}{\mathscr{V}}

\newcommand{\Xscr}{\mathscr{X}}
\newcommand{\Gr}{{\rm Gr}}

\def\GL{\operatorname{GL}}

\usepackage[addedmarkup=bf,deletedmarkup=sout]{changes}

\newcommand{\abs}[1]{\left| #1\right|}
\newcommand{\norm}[1]{\left\| #1\right\|}

\newcommand{\id}  {\operatorname{id}}

\newcommand{\comment}[1]{}
\newcommand{\minexp}{\mathfrak{m}}
\newcommand{\med}{\mathfrak{v}}

\newcommand{\blob}{\rule[.2ex]{.8ex}{.8ex}}

\newcommand{\graph}{{\rm Graph}}

\makeatletter
\newsavebox{\@brx}
\newcommand{\llangle}[1][]{\savebox{\@brx}{\(\m@th{#1\langle}\)}%
  \mathopen{\copy\@brx\mkern2mu\kern-0.9\wd\@brx\usebox{\@brx}}}
\newcommand{\rrangle}[1][]{\savebox{\@brx}{\(\m@th{#1\rangle}\)}%
  \mathclose{\copy\@brx\mkern2mu\kern-0.9\wd\@brx\usebox{\@brx}}}
\makeatother

\newcommand{\linspan}[1]{\llangle {#1} \rrangle}

\begin{document}
\title[Hyperbolic billiards on polytopes]{Hyperbolic billiards on polytopes with contracting reflection laws}

\date{\today}

\author[Duarte]{Pedro Duarte}
\address{Departamento de Matem\'atica and CMAF \\
Faculdade de Ci\^encias\\
Universidade de Lisboa\\
Campo Grande, Edifício C6, Piso 2\\
1749-016 Lisboa, Portugal 
}
\email{pmduarte@fc.ulisboa.pt}

\author[Gaiv\~ao]{Jos\'e Pedro Gaiv\~ao}
\address{Departamento de Matem\'atica and CEMAPRE, ISEG\\
Universidade de Lisboa\\
Rua do Quelhas 6, 1200-781 Lisboa, Portugal}
\email{jpgaivao@iseg.ulisboa.pt}

\author[Soufi]{Mohammad Soufi}
\address{Instituto de Matem\'atica \\
Universidade Federal de Alagoas\\
Campus A. C. Sim\~oes, Av. Lourival Melo Mota, s/n, 
57072-900 Macei\'o-AL, Brasil 
}
\email{msoufin@gmail.com}
\begin{abstract}
We study billiards on polytopes in $\Rr^d$ with contracting reflection laws, i.e. non-standard reflection laws that contract the reflection angle towards the normal. We prove that billiards on generic polytopes are uniformly hyperbolic provided there exists a positive integer $k$ such that for any $k$ consecutive collisions, the corresponding normals of the faces of the polytope where the collisions took place generate $\Rr^d$. As an application of our main result we prove that billiards on generic polytopes are uniformly hyperbolic if either the contracting reflection law  is sufficiently close to the specular or the polytope is obtuse. Finally, we study in detail the billiard on a family of $3$-dimensional simplexes. 
\end{abstract}

\maketitle


\section{Introduction}
\label{sec:introduction}

Given a $d$-dimensional polytope $P$, a billiard trajectory inside $P$ is a polygonal path described by a point particle moving with uniform motion in the interior of $P$. When the particle hits the interior of the faces of $P$, it bounces back according to a reflection law. Therefore, a billiard trajectory is determined by a sequence of reflections on the faces of $P$. Any reflection can be represented by a pair $x=(p,v)$ where $p$ is a point belonging to a face of $P$ and $v$ is a unit velocity vector pointing inside $P$. We denote by $M$ the set of reflections. The map $\Phi:M\to M,\, x\mapsto x'$ that takes a reflection $x$ to the next reflection $x'$ is called the \textit{billiard map}. The dynamics of billiards on polytopes has been mostly studied considering the specular reflection law. More recently, in the case of polygonal billiards, a new class of reflection laws has been introduced that contract the reflection angle towards the normal of the faces of the polygon \cite{arroyo09,markarian10,arroyo12,MDDGP12}. These are called \textit{contracting reflection laws}. A billiard map with a contracting reflection law is called a \textit{contracting billiard map}. It is known that strongly contracting billiard maps on generic convex polygons are uniformly hyperbolic and have finite number of ergodic SRB measures \cite{MDDGP13}. Recently, it has been proved that the same conclusion hods for contracting billiard maps on polygons with no parallel sides facing each other (even for contracting reflection laws close to the specular and for non-convex polygons) \cite{MDDG15}. 

In this paper we extend some of the previous results to contracting billiard maps on polytopes. It is known that the contracting billiard map of any polygon has dominated splitting~\cite[Proposition~3.1]{MDDGP13}. In this direction we show in Proposition~\ref{prop:partialhyperbolic} that the contracting billiard map of any polytope is always (uniformly) partially hyperbolic, i.e. there is a continuous and invariant splitting $E^s\oplus E^{cu}$ of the tangent bundle of $M$ into subbundles of the same dimension such that $D\Phi$ uniformly contracts vectors in the \textit{stable} subbundle $E^s$ and has neutral or expanding action on vectors belonging to the \textit{centre-unstable} subbundle $E^{cu}$. 

There are essentially two obstructions for the uniform expansion in  the centre-unstable subbundle $E^{cu}$. The first obstruction is caused by the billiard orbits that get trapped in a subset of faces of $P$ whose normals do not span the ambient space $\Rr^d$. When $P$ is a polygon ($d=2$), those orbits are exactly the periodic orbits of period two, i.e. orbits bouncing between parallel sides of $P$. In fact, when $P$ has no parallel sides the contracting billiard map is uniformly hyperbolic~\cite[Proposition~3.3]{MDDGP13}. 
 As another example let $P$ be a  $3$-dimensional prism
and consider a billiard orbit unfolding in some plane parallel to 
the prism's base. The normals to the faces along this orbit will span a plane and the billiard map behaviour transversal to this plane is neutral. This leads to an expansion failure in $E^{cu}$.

In order to circumvent this obstruction we had to consider a class of polytopes which have the property that for any subset of $d$ faces of $P$ the corresponding normals span $\Rr^d$. A polytope with this property is called \textit{spanning} (see Definition~\ref{def spanning}). In addition to being spanning, we suppose that the normals to the $(d-1)$-faces incident with any given vertex are linearly independent (see Definition~\ref{def generic polytopes}). Spanning polytopes with these properties are \textit{generic}.
In fact they form an open and dense subset having full Lebesgue measure in the set of all polytopes. 

The second obstruction to uniform expansion corresponds to the billiard orbits that spend a significant amount of time bouncing near the skeleton of $P$. To control the time spent near the skeleton we introduced the notion of \textit{escaping time}. Roughly speaking, the escaping time of $x\in M$ is the least positive integer $T=T(x)\in\Nn\cup\{\infty\}$ such that the number of iterates it takes for the billiard orbit of $x$ to leave a neighbourhood of the skeleton of $P$ is less than $T$ (see Definition~\ref{def escaping}).

With these notions we prove that the contracting billiard map has non-zero Lyapunov exponents for almost every point with respect to any given ergodic invariant measure.
 More precisely we prove:

\begin{theorem}\label{th measure hyperbolic}
If the contracting billiard map $\Phi$ of a generic polytope has an ergodic invariant probability measure $\mu$ such that $T$ is integrable with respect to $\mu$, then $\mu$ is hyperbolic.  
\end{theorem}

When the contracting billiard map $\Phi$ has bounded escaping time, then $\Phi$ is uniformly hyperbolic. 

\begin{theorem}\label{th:main2}
If the contracting billiard map $\Phi$ of a generic polytope has an invariant set $\Lambda$ such that $T$ is bounded on $\Lambda$, then $\Phi|_{\Lambda}$ is uniformly hyperbolic.
\end{theorem}

Theorems~\ref{th measure hyperbolic} and \ref{th:main2} follow from Theorem~\ref{main theorem} which gives a uniform estimate on the expansion along the orbit of every point which is $k$-generating (see Definition~\ref{def k-generating}). Being $k$-generating simply means that the face normals along any orbit segment of length $k$ span $\Rr^d$. 

The strategy to prove Theorem~\ref{main theorem} is the following. Consider the billiard orbit $x_n=(p_n,v_n)$, $n\geq0$ of a $k$-generating point $x_0=(p_0,v_0)\in M$. Denote by $\eta_n$ the inward unit normal of the face of $P$ where the reflection $x_n$ takes place. In some appropriate coordinates, known as Jacobi coordinates, the unstable space $E^u(x_0)$ is represented by the orthogonal hyperplane $v_0^{\perp}$. If the velocity $v_1$ is collinear with the normal $\eta_1$, then the action of the derivative $D\Phi$ on $E^{u}(x_0)$ is neutral. Otherwise, the map $D\Phi$ expands the direction $v_0^\perp\cap V_1$ where $V_1$ denotes the plane spanned by the velocities $v_0$ and $v_1$. Similarly, $D\Phi^2$ expands the directions contained in $v_0^\perp\cap V_2$ where now $V_2$ is generated by the velocities $v_0$, $v_1$ and $v_2$.  However, it may happen that the plane spanned by the velocities $v_1$ and $v_2$ is the same obtained from the span by the normals $\eta_1$ and $\eta_2$, thus implying that $\dim(v_0^{\perp}\cap V_2)=1$. This coincidence of the \textit{velocity front} with the \textit{normal front} is called a collinearity (see Definition~\ref{def collinearity}). 

If a collinearity never occurs and $x_0$ is $k$-generating then 
the map $D\Phi^k$  expands $d-1$ distinct directions in $v_0^{\perp}$. Although collinearities prevent  full expansion of the iterates $D\Phi^n(x_0)$  they have the good trait of synchronizing the velocity front with the normal front. After a collinearity every time a new face is visited the angle between the new velocity and the previous velocity front is always bounded away from zero. This happens because this velocity angle is related to the angle between the new normal  and  the previous normal front, and also because we assume the polytope to be spanning.
Consider now the  velocity front $V$ at some collinearity moment.
The previous property implies expansion of $D\Phi^n(x_0)$ transversal to $v_0^\perp \cap V$ after the collinearity moment. 
Choosing a minimal collinearity  (see Definition~\ref{def minimal collinearity}) in the orbit of $x_0$ we can also ensure the expansion of $D\Phi^n(x_0)$ along the velocity front up the collinearity moment.
Putting  these facts together,
if at some instant $t<k$ a minimal collinearity  occurs on   the orbit of $x_0$ then for $n\geq t + k$  we have full expansion of $D\Phi^n(x_0)$  on $E^u$.

Because we seek uniform expansion, one has to deal with $\delta$-collinearities instead (see Definition~\ref{def delta collinearity}). Moreover, since the set of orbits in $M$ is not compact (one has to remove from $M$ the orbits which hit the skeleton of $P$), $\delta$-collinearities are more easily handled in a bigger set called the \textit{trajectory space}. The trajectory space is compact and defined in a symbolic space which only retains the velocities and the normals of the faces of $P$ where the reflections take place (see Definition~\ref{def trajectory}). Finally, using compactness and continuity arguments we derive Theorem~\ref{hyperb:main} which gives a uniform estimate on the expansion along an orbit segment of length $2k$ of any $k$-generating point. Then Theorem~\ref{main theorem} follows immediately from Theorem~\ref{hyperb:main}. The crucial tool to prove Theorem~\ref{hyperb:main} is Lemma~\ref{main:lin:alg} which gives a uniform expansion estimate on compositions of linear maps. Since this lemma is formulated in more conceptual terms, we believe that the ideas therein might be of independent interest.

In section~\ref{sec:escape} we show that contracting billiards on polytopes have finite escaping time if either the contracting law is close to the specular or the polytope is obtuse. This together with Theorem~\ref{th:main2} prove the following corollaries.

\begin{corollary}
The contractive billiard map of a generic polytope with a contracting reflection law sufficiently close to the specular one is uniformly hyperbolic.
\end{corollary}

\begin{corollary}
The contracting billiard map of a generic obtuse polytope is uniformly hyperbolic.
\end{corollary}

The rest of the paper is organized as follows. In section~\ref{sec:statements} we introduce some notation and define the contracting billiard on a polytope. We also derive several properties of contracting billiards maps and rigorously state our main result. In section~\ref{sec:regular polygons} we show that polytopes on general position are generic and in section~\ref{sec:escape} we study the escaping time on polyhedral cones. Technical results concerning the expansion of composition of linear maps are proved in section~\ref{sec:unifexp}. In section~\ref{sec:proof} we prove our main results.
Finally, in section~\ref{sec:examples} we
study in detail the contracting billiard of a family of 3-dimensional simplexes.

\section{Definitions and Statements}
\label{sec:statements}
A half-space in $\Rr^d$ ($d\geq 2$) is any set of the form
$\{\, x\in\Rr^d\,\colon\, \langle x,v\rangle\leq c \,\}$,
for some non-zero vector $v\in\Rr^d$ and some real number $c\in\Rr$. A polyhedron is any finite intersection of half-spaces in $\Rr^d$. A polytope is a compact polyhedron.
We call dimension of a polyhedron to the dimension of the affine subspace that it spans.
Let $P\subset\Rr^d$ be a $d$-dimensional polytope. 

The billiard on $P$ is a dynamical system describing the linear motion of a point particle inside $P$. When the particle hits the boundary of $P$, it gets reflected according to a reflection law, usually the specular reflection law. In the following we rigorously define the billiard map $\Phi_{P}$ with the specular reflection law. But first, let us introduce some notation.

\subsection{Basic Euclidean Geometry}
\label{Euclidean Geometry}
Let $V$ and $V'$ be Euclidean spaces with 
$\dim V=\dim V'=d$.
Given a linear map $L\colon V\to V'$,
the maximum expansion of $L$ is the operator norm  
$$ \norm{L}:= \max \{ \norm{L(v)} \colon  v\in V, \norm{v}=1 \}  $$
 while the minimum expansion of $L$, defined by
$$ \minexp(L):= \min \{ \norm{L(v)} \colon  v\in V, \norm{v}=1 \}$$
is either $0$, when $L$ is non invertible, or else
$\minexp(L)=\norm{L^{-1}}^{-1}$.

We denote by $L^\ast\colon V'\to V$ the adjoint operator of $L\colon V\to V'$.
Recall that the {\em singular values} of $L$ are the eigenvalues
of the conjugate positive semi-definite symmetric operator 
$\sqrt{L^\ast\,L}$.
Being real, and non negative, the singular values of $L$
 can be ordered as follows
$$s_1(L)\geq s_2(L)\geq \ldots\geq s_d(L)\geq 0\;.$$
The top singular value is
$s_1(L)=\norm{L}$, while  the  last  singular value is the minimum expansion $s_d(L)=\minexp(L)$.
The product of all singular values of $L$ will be referred as the
{\em determinant} of $L$ 
$$ \det(L):=\prod_{j=1}^d s_j(L) . $$
This determinant is the factor by which   $L$ expands $d$-volumes.

Given $\lambda>0$ we denote by $\med^\geq _\lambda(L)$
the direct sum of all singular directions of $L$ (eigen-directions of $L^\ast\,L$) associated with singular values $\mu\geq \lambda$. Likewise, we denote by $\med^<_\lambda(L)$
the direct sum of all singular directions of $L$  associated with singular values $\mu<\lambda$.
It follows from these definitons that
\begin{align*}
V= \med^\geq _\lambda(L)\oplus  \med^<_\lambda(L),\; 
  L (\med^\geq_\lambda(L) )= \med^\geq_\lambda(L^\ast)\; \text{ and }\; 
L (\med^<_\lambda(L) ) = \med^<_\lambda(L^\ast) 
\end{align*}
and similar relations hold for $L^\ast$.
To shorten notations we will simply write
$\med(L)$ instead of $\med^{\geq}_{\norm{L}}(L)$.
This subspace will be referred to as
the {\em most expanding direction} of $L$.

\bigskip

Given vectors $v_1,\ldots, v_n\in\Rr^d$,
 the linear subspace spanned by the vectors
  $v_1,\ldots, v_n$ is denoted by $\linspan{v_1,\ldots, v_n}$.
Let $\Ss$ denote the unit sphere in $\Rr^d$, i.e. $\Ss=\{\, v\in\Rr^d\,:\, \norm{v}=1\,\}$. Let $v,\eta\in\Ss$ be unit vectors and $u\in\Rr^d$. We denote by $\Ss^+_\eta$ the \textit{hemisphere associated with $\eta$}, 
$$
\Ss^+_\eta:=\{\, v\in \Ss\,:\, \langle v,\eta\rangle >0 \,\}\;.
$$
Let $\eta^\perp$ denote the orthogonal hyperplane to $\eta$. The {\em orthogonal projection of $u$ onto the hyperplane $\eta^\perp$} is,
  $$P_{\eta^\perp}(u):= u- \langle u,\eta\rangle \, \eta = u - P_\eta(u)\;,$$
  where $P_\eta(u)=\langle u,\eta\rangle \, \eta$, is
  the \textit{orthogonal projection of $u$ onto the line spanned by $\eta$}. The \textit{reflection of $u$ about the hyperplane $\eta^\perp$} is defined by,
$$
R_{\eta}(u):=u-2\left\langle u,\eta\right\rangle\eta.
$$
Finally, the \textit{parallel projection of $u$ along $v$ onto the hyperplane $\eta^\perp$} is 
$$P_{v,\eta^\perp}(u):= u- \frac{\langle u,\eta\rangle}{\langle v,\eta\rangle} \, v\;.$$

\bigskip

Denote by $\angle(v,w)$ the angle between two non-zero vectors in $\Rr^d$, defined as
$$ \angle (v,w):= \arccos \left(\frac{ \langle v,w\rangle }{\norm{v} \norm{w} } \right) \;. $$
The angle between a  non-zero vector $v\in \Rr^d$ and a linear subspace $E\subseteq \Rr^d$ is defined
to be
$$ \angle (v,E):= \min_{u\in E\setminus\{0\}} \angle(v,u) \;. $$
The angle between two linear subspaces $E$ and $F$ of $\Rr^d$ of the same dimension is
defined as
$$ \angle (E,F):= \max \{ \, \max_{u\in E\setminus\{0\}} \angle(u,F),
\max_{v\in F\setminus\{0\}} \angle(v,E)\, \} \;. $$
This angle defines a metric on the Grassmann manifold
$\Gr_k(\Rr^d)$ of all $k$-dimensional linear subspaces $E\subseteq \Rr^d$.

Given two linear subspaces $E,F\subseteq \Rr^d$, with
$\dim E\leq \dim( F^\perp)$,  we define the minimum angle
$$ \angle_{\min}(E,F):= \min\left\{ \angle(e,f)\colon\,
e\in E\setminus \{0\}, \, f\in F\setminus \{0\} \, \right\} .$$
Unlike the previous angle, this minimum angle is not even a pseudo-metric on $\Gr(\Rr^d)=\cup_{0\leq k\leq d} \Gr_k(\Rr^d)$.
Notice that $\angle_{\min}(E,F)>0$ if and only if $E\cap F=\{0\}$.
The minimum angle $\angle_{\min}(E,F)$ quantifies the `transversality' on the intersection $E\cap F$.

We denote by $\pi_{E,F^\perp}:E\to F^\perp$ the restriction to $E$ of the orthogonal projection to  $F^\perp$.

\begin{lemma}
\label{sin angle(E,F)}
Given linear subspaces $E,F\subseteq \Rr^d$ with $\dim E=\dim F$,
$$ \sin \angle(E,F)= \norm{\pi_{E,F^\perp}} = \norm{\pi_{F,E^\perp}} . $$
\end{lemma}

\begin{proof}
Given $u\in E\setminus\{0\}$ and $v\in F\setminus\{0\}$, we have
\begin{enumerate}
\item  $ {\norm{\pi_{E,F^\perp}(u)}}/{\norm{u}}=   {d(u,F)}/{\norm{u}}= \sin\left( \angle(u,F) \right)$,
\item  $\norm{\pi_{F,E^\perp}(v)}/{\norm{v}} =    {d(v,E)}/{\norm{v}} = \sin\left(  \angle(v,E) \right) $.
\end{enumerate}
Since $\dim E=\dim F$, there is an orthogonal linear map $\Phi\colon \Rr^d\to\Rr^d$ such that $\Phi(F)=E$.
By orthogonality one has $\Phi(F^\perp)=E^\perp$. Hence
$\pi_{F,E^\perp} = \Phi^{-1} \circ \pi_{E,F^\perp} \circ \Phi$,
which implies that $\norm{\pi_{E,F^\perp}} = \norm{\pi_{F,E^\perp}}$.
Thus the sine of the maxima in the definition of $\angle(E,F)$ coincides with this common norm.
\end{proof}

\begin{lemma}\label{angle:lemma}
Let $E,E'$ and $H$ be linear subspaces of $\Rr^d$ such that 
\begin{enumerate}
\item $\dim(E)=\dim(E')$,
\item $\angle(h,E)\geq \varepsilon$, for all \, $h\in H\setminus\{0\}$.
\end{enumerate}
Then
$$ \sin\left( \angle( E+H,E'+H) \right) \leq 
\frac{ \sin\left( \angle( E,E') \right) }{\sin \varepsilon } $$
\end{lemma}

\begin{proof}
First notice that
$$ \angle( E+H,E'+H)= \angle( (E+H)\cap H^ \perp, (E'+H)\cap H^ \perp ) \;.$$
Given $u\in (E+H)\cap H^ \perp$ we can write $u=v-h$ with $v\in E$ and $h\in H$.
Hence, since $u\in H^\perp$,
\begin{align*}
\frac{ d(u, (E'+H)\cap H^ \perp)}{\norm{u}} &=
\frac{ d(u, E'+H)}{\norm{u}} = \frac{ d(v, E'+H)}{\norm{u}}\\
& \leq  \frac{ d(v, E')}{\norm{u}} 
= \frac{\norm{v}}{\norm{u}} \frac{ d(v, E')}{\norm{v}} \leq \frac{\norm{v}}{\norm{u}} \sin \left(\angle(E,E')\right)\\
& = \frac{\sin \left(\angle(E,E')\right)}{\sin\left(\angle(v,h)\right)} 
\leq \frac{\sin \left(\angle(E,E')\right)}{\sin\varepsilon} \;.
\end{align*}
On the last equality we use that $v=h+u$ is an orthogonal decomposition with $h\in H$ and $u\in H^ \perp$.
Thus taking the sup   in $u\in (E+H)\cap H^ \perp\setminus\{0\}$ we get 
$$ \sin\left( \angle( (E+H)\cap H^ \perp, (E'+H)\cap H^ \perp ) \right) \leq \frac{\sin \left(\angle(E,E')\right)}{\sin\varepsilon} \;.$$
\end{proof}

\begin{lemma}
\label{det<minexp}
Given linear subspaces $E,F\subseteq \Rr^d$
with $\dim E\leq \dim (F^\perp)$,
$$\det(\pi_{E,F^\perp})\leq \minexp (\pi_{E,F^\perp}).$$
Equality holds when $\dim E=1$.
\end{lemma}

\begin{proof} 
Just notice that all singular values of $\pi_{E,F^\perp}$ are in the range $[0,1]$ because $\pi_{E,F^\perp}$ is the restriction of an orthogonal projection.  
\end{proof}

Given an integer $k\in\Nn$ and a linear subspace $E\subseteq \Rr^d$,
the {\em Grassmann space} of $k$-vectors in $E$ will be denoted by $\wedge_k(E)$. This space inherits a natural Euclidean structure from $E$ (see~\cite{Stern}).

\begin{lemma}
\label{anglo prod exterior}
Let $E,F\subseteq \Rr^d$ be linear subspaces with
orthonormal basis $\{e_1,\ldots, e_k\}$ and
$\{f_1,\ldots, f_r\}$ respectively such that $k\leq d-r$.
Let $e=e_1\wedge \ldots \wedge e_k\in \wedge_k(E)$ and
$f=f_1\wedge \ldots \wedge f_r\in \wedge_r(F)$. Then 
$$\sin \angle_{\min}(E,F)= \minexp(\pi_{E,F^\perp} ) \geq \det(\pi_{E,F^\perp}) = \norm{e\wedge f} . $$
\end{lemma}

\begin{proof}
Given a unit vector $v\in E$, by the proof of Lemma~\ref {sin angle(E,F)} we have
$\sin \angle(v,F)= \norm{\pi_{E,F^\perp}(v)}$ which implies that
$$\sin \angle_{\min}(E,F)= \minexp(\pi_{E,F^\perp} ) . $$
On the other hand
\begin{align*}
\norm{e\wedge f} 
&= \norm{( e_1\wedge \ldots \wedge e_k) \wedge (f_1\wedge \ldots \wedge f_r)  } \\
&= \norm{( \pi_{E,F^\perp}(e_1)\wedge \ldots \wedge \pi_{E,F^\perp}(e_k)) \wedge (f_1\wedge \ldots \wedge f_r)  } \\
&= \norm{  \pi_{E,F^\perp}(e_1)\wedge \ldots \wedge \pi_{E,F^\perp}(e_k)}\,\norm{  f_1\wedge \ldots \wedge f_r   } \\
&= \norm{  \wedge_k \pi_{E,F^\perp}(e)}\,\norm{  f  } = \det(\pi_{E,F^\perp})
\end{align*}
because $\norm{e}=\norm{f}=1$.
The middle inequality follows from Lemma~\ref{det<minexp}.
\end{proof}

\begin{lemma}
\label{anglo cumulativo}
Let $E\subseteq \Rr^d$ be a linear space  
and $\{v_1,\ldots, v_k\}$ be a family of unit vectors such that for all $1\leq i\leq k$,
$$ \angle_{\min}(\linspan{v_i}, E\oplus \linspan{v_1,\ldots, v_{i-1}} )\geq \varepsilon >0 .$$
Then
$$ \sin \angle_{\min}(E,  \linspan{v_1,\ldots, v_{k}} )\geq    (\sin \varepsilon)^k   .$$
\end{lemma}

\begin{proof}
Let $\{e_1,\ldots, e_r\}$ be an orthonormal basis of $E$.
We apply Lemma \ref{anglo prod exterior}
 to the subspaces $\linspan{v_i}$ and $E\oplus \linspan{v_1,\ldots, v_{i-1}}$. Since the first subspace has dimension $1$ the inequality in this lemma becomes an equality. Hence, because $\norm{v_i}=1$ we have
$$ \frac{\norm{e_1\wedge \ldots \wedge e_r \wedge v_1 \wedge \ldots
\wedge  v_{i}} }{\norm{e_1\wedge \ldots \wedge e_r \wedge v_1 \wedge \ldots
\wedge v_{i-1}}}\geq \sin\varepsilon $$
for all $1\leq i \leq k$.
Multiplying these inequalities and using Lemma~\ref{anglo prod exterior} again we obtain
\begin{align*}
\sin \angle_{\min}(E,  \linspan{v_1,\ldots, v_{k}} ) &\geq  
\frac{\norm{e_1\wedge \ldots \wedge e_r \wedge v_1 \wedge \ldots
\wedge  v_{k}}}{\norm{e_1 \wedge \ldots
\wedge  e_{r}}\,\norm{v_1 \wedge \ldots
\wedge  v_{k}}} \\
&\geq  
\frac{\norm{e_1\wedge \ldots \wedge e_r \wedge v_1 \wedge \ldots
\wedge  v_{k}}}{\norm{e_1 \wedge \ldots
\wedge  e_{r}}} \\
&=\prod_{i=1}^k \frac{\norm{e_1\wedge \ldots \wedge e_r \wedge v_1 \wedge \ldots
\wedge  v_{i}} }{\norm{e_1\wedge \ldots \wedge e_r \wedge v_1 \wedge \ldots
\wedge v_{i-1}}}\geq (\sin \varepsilon)^k .
\end{align*}
We have used above that 
$\norm{v_1\wedge \ldots \wedge v_k}\leq \norm{v_1} \cdots \norm{v_k} = 1$.
\end{proof}

\subsection{Billiard map}
Suppose that $P$ has $N$ faces (of dimension $d-1$) which we denote by $F_1,\ldots, F_N$. For each $i=1,\ldots, N$, denote by $\eta_i$ the interior unit normal vector to the face $F_i$. Also denote by $\Pi_i$ the hyperplane that supports the face $F_i$. We write the interior of $F_i$ as $F_i^\circ$,
and its $(d-2)$-dimensional boundary as $\partial F_i$.
Define $\partial P=\bigcup_{i=1}^N F_i$,
and the $(d-2)$-skeleton $\Sigma P= \bigcup_{i=1}^N \partial F_i$. 
Finally define
$$ M:=\bigcup_{i=1}^N F_i^\circ\times \Ss^ +_{\eta_i}\;.$$
The domain of the billiard map $\Phi_P$ is the set of points $(p,v)\in M$
such that the half-line $\{\, p+t\,v\,:\, t\geq 0\}$ does not intersect the skeleton
$\Sigma P$. We denote this set by $M'$. Clearly, $M'$ is the complement of a co-dimension two subset of $M$. 

Now the billiard map $\Phi_P:M'\to M$ is defined as follows. Given $x=(p,v)\in M'$, let $\tau=\tau(p,v)>0$ be minimum $t>0$ such that 
$p+t \, v\in F'_j$ for some $j=1,\ldots,N$. The real number $\tau$ is called the \textit{flight time} of $(p,v)$. Then the billiard map is defined by
$$\Phi_{P}(x)=(p+\tau\, v,  R_{\eta_j}(v) ).$$

Note that the billiard map $\Phi_P$ is a piecewise smooth map and it has finitely many domains of continuity. The number of domains of continuity is at most $N(N-1)$, which is the number of $2$-permutations of $N$ faces. If $P$ is convex, then all permutations define a branch map. 

Let $(p',v')=\Phi_P(p,v)$ for $(p,v)\in M'$. It is easy to obtain a formula for the branch maps and its derivatives. 

\begin{proposition}
\label{prop:DPhi}
Suppose that $(p'_i,v'_i)=\Phi_P(p_i,v_i)$ for some $p_i\in F_i^\circ$ such that $p'_i\in F'_j$ with $i\neq j$. For every $x=(p,v)\in F_i^\circ\times \Ss_{\eta_i}^+$ such that $p'\in F'_j$ we have
$$ \Phi_P(x)=\left( p_j+P_{v,\eta_j^\perp}(p-p_j), R_{\eta_j}(v)\,\right)\;.$$
Moreover
$$ D\Phi_P(x)(u,w)=\left( \,  P_{v,\eta_j^\perp}(u+\gamma(x)\,w), R_{\eta_j}(w)\,\right)\;,$$
where 
$$
\gamma(x)=\frac{\langle p-p_j,\eta_j \rangle}{\langle v,\eta_j \rangle}\,. 
$$

\end{proposition}

\begin{proof}
Recall that $p'=p+\tau(p,v)v$ where $\tau(p,v)$ is the length of the vector $p'-p$. Taking the inner product with $\eta_j$ in both sides of the equation and noting that $\langle p'-p_j,\eta_j \rangle=0$, we get
$$
\tau(p,v)=\frac{\langle p'-p,\eta_j \rangle}{\langle v,\eta_j \rangle}=\frac{\langle p_j-p,\eta_j \rangle}{\langle v,\eta_j \rangle}\,.
$$

So
$$
p'=p_j+\left((p-p_j)-\frac{\langle p-p_j,\eta_j \rangle}{\langle v,\eta_j \rangle} v\right)=p_j+P_{v,\eta_j^\perp}(p-p_j)\,.
$$
To prove the formula for the derivative, define the map $\Psi_{\eta}:(p,v)\mapsto P_{v,\eta^\perp}(p)$ for any given $\eta\in \Ss$. The claim follows from the formula
$$
D\Psi_\eta(x)(u,w)=P_{v,\eta^\perp}(u)+\frac{\langle p,\eta \rangle}{\langle v,\eta \rangle}P_{v,\eta^\perp}(w).
$$
\end{proof}

\subsection{Contracting  reflection laws}
\label{subsec: contractive laws}
A {\em contracting law} is any family 
$\{ \, C_\eta:\Ss^+_\eta\to \Ss^+_\eta\,\}_{\eta\in\Ss}$ of class $C^2$ mappings
that satisfies for every $\eta\in\Ss$,
\begin{enumerate}
\item[(a)] $C_\eta(\eta)=\eta$,
\item[(b)] there are non-negative $C^2$ functions $a_\eta,b_\eta:\Ss^+_\eta\to[0,+\infty)$ such that,
$$
C_\eta (v) = a_\eta(v)P_\eta(v)+b_\eta(v)P_{\eta^\perp}(v),\quad\forall\,v\in \Ss^+_\eta.
$$
\item[(c)] $0<\sup\{\, \norm{DC_\eta(x)}\,:\,  x\in \Ss^+_\eta\,\} < 1$, 
\item[(d)] $O\circ C_\eta = C_{O(\eta)}\circ O$, for every rotation $O\in\OO(n,\Rr)$.
\end{enumerate}

A contracting law can be uniquely characterized by a single $C^2$ map of the interval $\left[0,\frac{\pi}{2}\right)$ as the following proposition shows.

\begin{proposition} 
\label{prop contractive law}
Given a  contracting law
$\{ \, C_\eta:\Ss^+_\eta\to \Ss^+_\eta\,\}_{\eta\in\Ss}$,
there is a class $C^2$ mapping
$f:\left[0,\frac{\pi}{2}\right)\to \left[0,\frac{\pi}{2}\right)$
such that
\begin{enumerate}
\item[(a)] $f(0)=0$,  
\item[(b)] $0<\sup\{\, |f'(\theta)| \,:\,  0\leq\theta<\frac{\pi}{2}\,\} < 1$, 
\item[(c)] for every $\eta\in\Ss$, and $v\in\Ss^ +_\eta$,
$$
C_\eta(v) = \frac{\cos f(\theta) }{\cos\theta}P_\eta(v) + 
\frac{\sin f(\theta) }{\sin\theta}P_{\eta^\perp}(v)\;,
$$
where
$\theta  =\arccos \langle v,\eta\rangle$ is the angle between $\eta$ and $v$,
\item[(d)] for every $\eta\in\Ss$,
$$\sup_{x\in \Ss^+_\eta} \norm{DC_\eta(x)} = \sup_{0\leq\theta<\pi/2} |f'(\theta)|\;.$$
\end{enumerate}
\end{proposition}

\begin{proof}
Let $\eta\in\Ss$ and $v\in \Ss_\eta^+$. By item (b) of the definition of a contracting law we can write 
$$
C_\eta (v) = a_\eta(v)P_\eta(v)+b_\eta(v)P_{\eta^\perp}(v)
$$
where $a_\eta$ and $b_\eta$ are non-negative $C^2$ functions. Taking the inner product with $\eta$ on both sides of the previous equation we get,
$$
a_\eta (v) = \frac{\langle C_\eta v,\eta\rangle}{\cos\theta},
$$
where $\theta=\arccos \langle v,\eta\rangle\in[0,\frac\pi2)$ is the angle formed by the vectors $v$ and $\eta$. By item (d) we conclude that $\langle C_\eta(v),\eta\rangle = \langle C_{O(\eta)} ( O(v)),O(\eta)\rangle$, thus its value depends only on the angle $\theta$. So, there is a $C^2$ function $f:[0,\frac\pi2)\to[0,\frac\pi2)$ such that $\langle C_\eta(v),\eta\rangle=\cos f(\theta)$. Similarly, we conclude that 
$$
b_\eta(v) = \frac{\sin f(\theta)}{\sin\theta}.
$$
This shows (c). The remaining properties follow immediately.
\end{proof}

A $C^2$ mapping $f:\left[0,\frac{\pi}{2}\right)\to \left[0,\frac{\pi}{2}\right)$ satisfying (a)-(d) above is called a \textit{contracting reflection law}. We also define
$$
\lambda(f):=\sup_{0\leq\theta<\pi/2} |f'(\theta)|.
$$

\subsection{Contracting billiard map}
Given a contracting law $\{C_\eta\}$ with contracting reflection law $f$, define the map $\chi_f:M\to M$ by $\chi_f(p,v)=(p,C_{\eta(p)}(v))$ where $\eta(p)$ denotes the interior unit normal of the face of the polytope where $p$ lies. The \textit{contracting billiard map} $\Phi_{f,P}:M'\to M$ is 
$$
\Phi_{f,P}=\chi_f\circ\Phi_P.
$$
There is a system of coordinates which is convenient to represent the derivative of the contracting billiard map.
For each $x=(p,v)\in M$ define $\Psi_{x}:T_{x} M\to v^\perp\times v^\perp$ by
$$ \Psi_x(u,w)= \left(P_{v^\perp}(u), w\right)\;. $$
The previous linear isomorphism will be referred as {\em Jacobi coordinates}
on the tangent space $T_{x}M$.
We shall use the notation  $(J,J')$ to denote an element in $v^\perp\times v^\perp$.
The following proposition gives a formula for the derivative of the contracting billiard map
in terms of Jacobi co-ordinates.

\begin{proposition} 
\label{prop derivative in Jacobi coordinates}
Let $x=(p,v)\in M'$ and suppose that $x'=(p',v')=\Phi_{f,P}(x)$ with $p'\in F'_j$. Then 
$\Psi_{x'}\circ D\Phi_{f,P}(x)\circ \Psi_{x}^{-1}$ is given by
$$ (J,J')\mapsto\left( 
\,  P_{v'^ \perp} \circ P_{v,\eta_j^\perp}(J+\tau(p,v) \,J'),\,  
(D C_{\eta_j})_{ R_{\eta_j}(v)} R_{\eta_j}(J') 
\right) \;.$$
Moreover, if $\theta=\arccos |\langle v,\eta_j\rangle|$, then
$$ \left|\frac{\langle v',\eta_j\rangle}{\langle v,\eta_j\rangle}\right|=\frac{\cos f(\theta)}{\cos\theta} >1\;. $$
\end{proposition}

\begin{proof}
Immediate from Propositions \ref{prop:DPhi} and \ref{prop contractive law}.
\end{proof}

\subsection{Orbits, invariant sets and hyperbolicity}
\label{section hyperbolicity}
Denote by $M^+$ the subset of points in $M$ that can be iterated forward, i.e. 
$$
M^+=\{x\in M\colon \Phi_{f,P}^n(x)\in M'\;\forall\,n\geq 0\}.
$$
A \textit{billiard orbit} is a sequence $\{x_n\}_{n\geq0}$ in $M'$ such that $x_{n+1}=\Phi_{f,P}(x_n)$ for every $n\geq 0$. A \textit{billiard path or trajectory} is the polygonal path formed by segments of consecutive points of a billiard orbit. 

Define
$$
D:=\bigcap_{n\geq 0} \Phi_{f,P}^n(M^+).
$$
It is easy to see that $D$ is an invariant set and $\Phi_{f,P}$ and its inverse are defined on $D$. Following Pesin we call the closure of $D$ the \textit{attractor} of $\Phi_{f,P}$.
We say that $\Lambda\subset M$ is an \textit{invariant set} if $\Lambda\subset D$ and $\Phi_{f,P}^{-1}(\Lambda)=\Lambda$. 

To simplify the notation let us write $\Phi=\Phi_{f,P}$. 

\begin{defn}
Given an invariant set $\Lambda$ of $\Phi$, we say that $\Phi$ is  \textit{uniformly partially hyperbolic on $\Lambda$} if for every $x\in\Lambda$ there exists a continuous splitting 
$$
T_xM=E^s(x)\oplus E^{cu}(x),
$$
which is $D\Phi$-invariant,
$$
\quad D\Phi(E^s(x))=E^s(\Phi(x)),\quad D\Phi(E^{cu}(x))=E^{cu}(\Phi(x)),\quad\forall\,x\in\Lambda
$$
and there are constants $\lambda<1$, $\sigma\geq 1$ and $C>0$ such that for every $n\geq 1$ we have
 $$
 \| D\Phi^n|_{E^s}\|\leq C \lambda^n\quad\text{and}\quad \| D\Phi^{-n}|_{E^{cu}}\|\leq C \sigma^{-n}.
 $$
If $\sigma>1$, then we say that $\Phi$ is \textit{uniformly hyperbolic on $\Lambda$} and write $E^{u}$ for the subbundle $E^{cu}$. When $\Lambda=D$, then we simply say that $\Phi$ is \textit{uniformly partially hyperbolic}.
\end{defn}

We denote by
$$
\chi(x,v)=\limsup_{n\to\infty}\frac{1}{n}\log\|D\Phi^n(x)v\|
$$
the \textit{Lyapunov exponent} of a non-zero tangent vector $v\in T_x M$ at $x\in D$.

\begin{defn}
A $\Phi$-invariant Borel probability measure $\mu$ supported on $D$ is called \textit{hyperbolic} if $\mu$-almost every $x\in D$ satisfies $\chi(x,v)\neq 0$ for every non-zero $v\in T_xM$.
\end{defn}

The proof of the following result is an adaptation of \cite[Proposition 3.1]{MDDGP13}.

 \begin{proposition}\label{prop:partialhyperbolic}
For any polytope $P$ and any contracting reflexion law $f$, $\Phi_{f,P}$ is uniformly partially hyperbolic.
\end{proposition}

\begin{proof}
Given $x=(p,v), x'=(p',v')\in M$ such that $x'=\Phi(x)$ we denote by
$L(x,x')$ the map from $v^\perp \times v^\perp$ to 
 $v'^\perp \times v'^\perp$ that represents the derivative $D\Phi_x$ in the Jacobi coordinates (see Proposition~\ref{prop derivative in Jacobi coordinates}). This linear map is represented by a block upper triangular matrix of the form 
$$ L(x,x') = \begin{pmatrix}
 A(x,x') & B(x,x') \\ 0 & C(x,x') 
 \end{pmatrix}$$
where $\norm{A(x,x')^{-1}}\leq 1$ and 
$\norm{C(x,x')}\leq \lambda <1$, whose inverse is 
$$ L(x,x')^{-1} = \begin{pmatrix}
 A^{-1} & - A^{-1} B C^{-1} \\ 0 & C^{-1} 
 \end{pmatrix}$$
where $A=A(x,x')$, etc.
Given a linear map $H':v'^\perp \to v'^\perp$ the pre-image
of its graph by $L(x,x')$ is the graph of another linear function
$H:v^\perp \to v^\perp$ called the {\em backward graph transform} of $H'$
and denoted by $H=:\Gamma(x,x') H'$. The operator
$\Gamma(x,x')$ is hence defined by the relation
$$ L(x,x')^{-1}\graph (H')= \graph\left( \Gamma(x,x') H' \right).$$
A simple computation shows that
$$ \Gamma(x,x') H' = A(x,x')^{-1} B(x,x') - A(x,x')^{-1} H' C(x,x') .$$
We claim that writing $x_n=(p_n,v_n)=\Phi^n x$ and 
denoting by $Z_n$ the zero endomorphism  on $v_n^\perp$,
the following limit exists
$$ H^s(x) := \lim_{n\to +\infty}
\Gamma(x,\Phi x) \ldots \Gamma(\Phi^{n-1} x,\Phi^n x)  Z_n . $$
A recursive computation allows to explicit  the right hand side composition
$\Gamma(x,\Phi x) \ldots \Gamma(\Phi^{n-1} x,\Phi^n x)  Z_n$,
which is a partial sum of the following series
$$ H^s(x)  =  \sum_{j=0}^\infty
(-1)^j A_0^{-1}\cdots  A_j^{-1} B_j C_{j-1} \cdots C_0 $$ 
where
$A_j=A(\Phi^j x , \Phi^{j+1} x)$, etc.
This series converges because $\norm{A_j^{-1}}\leq 1$
and $\norm{C_j}\leq \lambda(f) < 1$ for all $j\geq 0$.

By construction, the subspaces $E^s(x):=\Psi_x^{-1}\graph(H^s(x))$
determine a $D\Phi$-invariant subbundle of $TM$  satisfying $\norm{D\Phi\vert_{E^s(x)}}\leq \lambda(f)$ for all $x\in D$.
Given $x=(p,v) \in D$, define $E^{cu}(x):=\Psi_x^{-1} \{(J,0)  \colon J\in v^\perp \} $. Clearly, $E^{cu}$ is invariant. Moreover, $\norm{ D\Phi^{-1}|_{E^{cu}(x)} }\leq   1$ for all $x\in D$. 

Finally, since $T_x M= E^s(x)\oplus E^{cu}(x)$ the previous facts show that $\Phi$ is uniformly partially hyperbolic.
\end{proof}

\subsection{Main results}

\begin{defn}
\label{def k-generating}
Given $k\in\Nn$, we say that $x\in M^+$ is \textit{$k$-generating} if the face normals along any orbit segment of length $k$ of the orbit of $x$ generate the Euclidean space $\Rr^d$. 
\end{defn}

\begin{defn} \label{def spanning}
Given $\varepsilon>0$,
 the polytope $P$ is called
\textit{$\varepsilon$-spanning} if for any $d$ distinct faces $F_{i_1},\ldots, F_{i_d}$ of $P$   with interior normals $\eta_{i_1},\ldots, \eta_{i_d}$, the angle between $\eta_{i_1}$ and $E:=\linspan{\eta_{i_2},\ldots, \eta_{i_d}}$ is at least $\varepsilon$, i.e.
$$\angle(\eta_{i_1},E)\geq \varepsilon.$$
We also say that $P$ is a \textit{spanning polytope} if it is $\varepsilon$-spanning for some $\varepsilon>0$.
\end{defn}

The following theorem is the main result of this paper. It shows that the contracting billiard map uniformly expands the unstable direction along the orbit of any $k$-generating point. Moreover, the expanding rate only depends on the polytope and contracting reflection law.

\begin{theorem}
\label{main theorem}
Suppose $P$ is a spanning polytope and $f$ a contracting reflexion law.  There exists $\sigma=\sigma(f,P)>1$, depending only on $f$ and  $P$, such that for every $k\geq d$ and every $k$-generating $x\in D$,
$$
\| D\Phi_{f,P}^{-2k}|_{E^u(x)}\|\leq 1/\sigma .
$$
\end{theorem}

We prove this theorem and the following results in section~\ref{sec:proof}.

\begin{defn}\label{def escaping}
Given $x\in M^+$, the \textit{escaping time of $x$}, denoted by $T(x)$, is the least positive integer $k\in\Nn$ such that $x$ is $k$-generating. If $x$ is not $k$-generating for any $k\in \Nn$, then we set $T(x)=\infty$. We also call the function $T:M^+\to\Nn\cup\{\infty\}$ the \textit{escaping time of $P$ with respect to $f$}.
\end{defn}

\begin{theorem}
\label{coro NUH}
Suppose $P$ is a spanning polytope and $\mu$ is an ergodic $\Phi_{f,P}$-invariant Borel probability measure. If $T$ is $\mu$-integrable, then $\mu$ is hyperbolic.
\end{theorem}

\begin{theorem}
\label{coro UH}
Suppose $P$ is a spanning polytope and $\Lambda$ an invariant set of $\Phi_{f,P}$. If $T$ is bounded on $\Lambda$, then $\Phi_{f,P}$ is uniformly hyperbolic on $\Lambda$.
\end{theorem}

 The concept of polytope in general position, mentioned in the following corollaries,  is defined below  (see definition~\ref{def generic polytopes}).

\begin{corollary}\label{coro 1}
Suppose $P$ is a polytope in general position. There exists $\lambda_0=\lambda_0(P)>0$ such that for every contracting reflection law $f$ satisfying $\lambda(f)>\lambda_0$ the billiard map $\Phi_{f,P}$ is uniformly hyperbolic. 
\end{corollary}

A polytope $P$ in general position is called \textit{obtuse} if the barycentric angle at every vertex of $P$ is greater than $\pi/4$ (see section~\ref{sec:escape} for a precise definition).

\begin{corollary}\label{coro 2}
Suppose $P$ is a polytope in general position and $f$ any contracting reflection law. If $P$ is obtuse, the $\Phi_{f,P}$ is uniformly hyperbolic.
\end{corollary}

\section{Generic Polytopes}
\label{sec:regular polygons}

\begin{defn}
\label{def generic polytopes}
A $d$-dimensional polytope  $P$ is said to be in general position if 
\begin{enumerate}
\item for any set of $d$ faces of $P$, $(d-1)$-dimensional faces,  their normals are linearly independent,
\item the normals to the  $(d-1)$-faces  of $P$ incident with any given vertex are linearly independent.
\end{enumerate}
\end{defn}

\begin{proposition}
Given some $d$-dimensional polytope  $P\subset\Rr^d$  in general position,
each vertex has exactly $d$ faces and $d$ edges incident with it.
\end{proposition}

\begin{proof}
Follows from condition (2) of the Definition~\ref{def generic polytopes}.
\end{proof}

Consider the class $\mathcal{P}_{N}$ of $d$-dimensional
polyhedra $P\subset\Rr^d$   that contain the origin, i.e., $0\in {\rm int} (P)$,
with exactly $N$ faces.
Given $N$ points $(p_1,\ldots, p_N)\in (\Rr^d\setminus \{0\})^N$, define
the polytope $Q(p_1,\ldots, p_N)\subset\Rr^d$,
$$ Q(p_1,\ldots, p_N) :=\cap_{j=1}^N \{\, x\in\Rr^b\,\colon\,
\langle x, p_j\rangle \leq \langle p_j, p_j\rangle\,\}\;. $$
The set
$$ \mathcal{U}:=\{\, (p_1,\ldots, p_N)\in (\Rr^d\setminus \{0\})^N\,\colon\,
Q(p_1,\ldots, p_N)\; \text{ has exactly }\, N\text{-faces}\;\} $$
is open in $(\Rr^d\setminus \{0\})^N$, and the range of $Q:\mathcal{U}\to \cdot$
coincides with $\mathcal{P}_{N}$.
Locally the map $Q:\mathcal{U}\to \mathcal{P}_{N}$ is one-to-one, and determines
an atlas for a smooth structure on $\mathcal{P}_{N}$.
We will consider on this manifold the Lebesgue measure obtained as 
push-forward of the Lebesgue measure  on $(\Rr^d\setminus \{0\})^N$ by the map $Q$.

Let $\mathscr{P}_{N}$ denote the subset of   polytopes in $\mathcal{P}_{N}$.

In Algebraic Geometry, the following result is a standard consequence of the notion of `general position'.
We include its proof here for the reader's convenience, also because we could not find any reference for this precise statement.

\begin{proposition}
The subset of  polytopes in general position is  
is open and dense, and has full Lebesgue measure in $\mathscr{P}_{N}$.
\end{proposition}

\begin{proof}
Consider the subsets $\mathcal{N}_1\subset \mathcal{P}_{N}$, resp. $\mathcal{N}_2\subset \mathcal{P}_{N}$, of polytopes where condition (1), resp. (2), of definition~\ref {def generic polytopes} is violated. It is enough to observe that the sets $\mathcal{N}_1$ and  $\mathcal{N}_2$
 are finite unions of algebraic varieties of co-dimension one.

 For any vector $v=(v_1,\ldots, v_d) \in \Rr^d$, let
 $\hat v:=(v_1,\ldots, v_d, \langle v,v\rangle)\in\Rr^{d+1}$. Then $\mathcal{N}_2$ is covered by the union over all $1\leq i_1 <i_2 <\ldots < i_{d+1}\leq N$
 of the  hypersurfaces defined by the algebraic equation
\begin{equation}
\label{det hat p = 0}
\det [ \, \hat p_{i_1},  \hat p_{i_2}, \ldots ,
  \hat p_{i_{d+1}} ]=0\;. 
\end{equation}
 In fact, if there is a point $x_0\in\Rr^d$ in the intersection of $d+1$ distinct hyperplanes
$$ \langle p_{i_k}, x\rangle = \langle p_{i_k}, p_{i_k}\rangle\quad k=1,\ldots, d+1 $$
then the matrix with rows $\hat p_{i_1},  \hat p_{i_2}, \ldots ,
  \hat p_{i_{d+1}}$ contains the vector $(x_0,-1)\in\Rr^{d+1}$ in its kernel,   which implies~\eqref{det hat p = 0}.

Analogously, $\mathcal{N}_1$ is contained in the union over all $1\leq i_1 <i_2 <\ldots < i_{d}\leq N$
 of the  hypersurfaces defined by the algebraic equation
 $$ \det [ \,   p_{i_1},    p_{i_2}, \ldots ,
    p_{i_{d}} ]=0\;. $$
\end{proof}

\section{Escaping Times}
\label{sec:escape}

In this section we study the escaping times of billiards on polyhedral cones with contracting reflection laws.

Let $\Pi_1,\ldots,\Pi_s$ be $s$ hyperplanes in $\Rr^d$ passing through the origin. For each hyperplane $\Pi_i$ we take a unit normal vector $\eta_i$ and we suppose that the set of hyperplanes are in general position, i.e. the normal vectors $\eta_1,\ldots,\eta_s$ are linearly independent. A set of $s$ hyperplanes in general position define a \textit{convex polyhedral cone}
$$
Q=\{x\in \Rr^d\colon \left\langle x,\eta_i\right\rangle\geq0\,,\quad i=1,\ldots,s\}\,.
$$

For polyhedral billiard with the specular reflection law, Sinai proved that there exists a constant $K>0$, depending only on $Q$, such that every billiard trajectory in $Q$ has at most $K$ reflections \cite{Sinai78}. In this case we say that $Q$ has \textit{finite escaping time}. 

By projecting the billiard dynamics to the orthogonal complement of $\bigcap_{i=1}^s\Pi_i$, we may assume that the normal vectors $\eta_1,\ldots,\eta_s$ defining the polyhedral cone $Q$ span $\Rr^d$. Thus, from now on we set $s=d$. 
Associated with a convex polyhedral cone $Q$ there is a constant measuring the aperture of $Q$. It is defined as follows. 
\begin{defn}
\label{def barycentric angle}
The normal vectors $\eta_1,\ldots,\eta_d$ regarded as points determine a affine hyperplane $H$ and a unit normal vector $e$ such that
$$
\left\langle \eta_i,e\right\rangle=\ell\,,\quad i=1,\ldots,d\,,
$$
where $\ell$ is the distance of $H$ to the origin. The \textit{barycentric angle} $\phi$ of $Q$ is defined by $\sin\phi=\ell$ (see Figure~\ref{angle}). Note that $0<\phi<\pi/2$. We say that a convex polyhedral cone $Q$ is \textit{obtuse} if $\phi>\pi/4$.
\end{defn}

\begin{figure}[h]
\includegraphics[width=7cm]{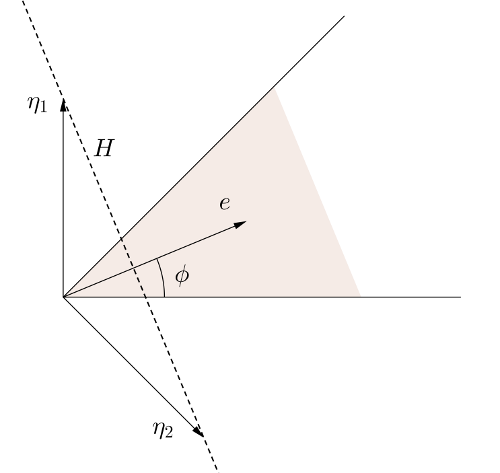}
\caption{Barycentric angle $\phi$.}
\label{angle}
\end{figure}

\subsection{Zigzag reflections}
According to Proposition \ref{prop contractive law}, given any billiard orbit $\{(p_k,v_k)\}_{k\geq 0}$, the sequence of \textit{reflection velocities} satisfies
\begin{equation}\label{eq:vk}
v_{k+1}=\frac{\cos f(\theta_k)}{\cos\theta_k}P_{\eta_{i_k}}(u_k)+\frac{\sin f(\theta_k)}{\sin\theta_k}P_{\eta_{i_k}^\perp}(u_k)\,,\quad k\geq0,
\end{equation}
where $u_k=R_{\eta_{i_k}}(v_k)$, $\theta_k=\arccos\left\langle u_k,\eta_{i_k}\right\rangle$ and $\eta_{i_k}$ is the inward normal of $P$ where the $k+1$-th collision took place.

\begin{lemma}\label{lem:vkdiff}$\|v_{k+1}-v_k\|=2\cos\left(\frac{f(\theta_k)+\theta_k}{2}\right)$ for every $k\geq 0$.
\end{lemma}

\begin{proof}
Simple computation using \eqref{eq:vk}.
\end{proof}

Given a sequence of consecutive reflection velocities $v_0,\ldots,v_n$ we denote by $L$ the length of the zigzag path formed by the reflections, i.e.
$$
L(v_0,\ldots,v_n)=\sum_{k=0}^{n-1}\|v_{k+1}-v_k\|\,.
$$
We say that $Q$ has \textit{bounded zigzag reflections} if there exists a constant $C>0$ such that $L(v_0,\ldots,v_n)\leq C$ for every sequence of consecutive reflection velocities $v_0,\ldots,v_n$ and any $n\geq 0$.

\begin{lemma} \label{prop:escapingzigzag}
A convex polyhedral cone has finite escaping time if and only if it has bounded zigzag reflections.
\end{lemma}
\begin{proof}
If $Q$ has finite escaping time, then there exists an integer $K>0$ such that every billiard trajectory has at most $K$ reflections. Since the zigzag length $L:\prod_{i=1}^K\Ss^m\to \Rr$ is a continuous function with compact domain, it has a maximum. Thus, $Q$ has bounded zigzag reflections.

Now suppose that $Q$ has not finite escaping time. This means that for every $K>0$ there exists a billiard trajectory in $Q$ that has at least $K$ reflections with the faces of $Q$. 
By Lemma~\ref{lem:vkdiff} we have
$\|v_{k+1}-v_k\|\geq \delta>0$ where $\delta:=2\cos\left(\frac{f(\pi/2)+\pi/2}{2}\right)>0$. This means that for every $K>0$ there exists a sequence of consecutive reflection velocities $v_0,\ldots,v_n$ such that $L(v_0,\ldots,v_n)\geq \delta K$. So $Q$ cannot have bounded zigzag reflections.
\end{proof}

Next we provide a sufficient condition on the contracting reflection law that guarantees boundedness of zigzag reflections. Thus finite escaping time.

\begin{lemma} \label{lem:vk}
For every sequence of consecutive reflection velocities $v_0,\ldots,v_n$ we have
$$
\left\langle v_{k+1}-v_k,e\right\rangle=\|v_{k+1}-v_k\|\gamma_k\,,\quad k=0,\ldots,n
$$
where 
$$
\gamma_k=\cos\left(\frac{f(\theta_k)-\theta_k}{2}\right)\sin\phi+\sin\left(\frac{f(\theta_k)-\theta_k}{2}\right)h_k
$$
and $h_k=\left\langle P_{\eta_{i_k}^\perp}(u_k)/\sin\theta_k,e\right\rangle$.
\end{lemma}

\begin{proof}
Follows from \eqref{eq:vk} that
$$
v_{k+1}-v_k=\frac{\cos{f(\theta_k)}+\cos\theta_k}{\cos\theta_k} P_{\eta_{i_k}}(u_k)+\frac{\sin{f(\theta_k)}-\sin\theta_k}{\sin\theta_k} P_{\eta_{i_k}^\perp}(u_k)\,.
$$
Taking into account that $P_{\eta_{i_k}}(u_k)/\cos\theta_k=\eta_{i_k}$ and $\left\langle \eta_{i_k},e\right\rangle=\sin\phi$ we get
$$
\left\langle v_{k+1}-v_k,e\right\rangle=\left(\cos{f(\theta_k)}+\cos\theta_k\right)\sin\phi+\left(\sin{f(\theta_k)}-\sin\theta_k\right)h_k\,,
$$
where $h_k=\left\langle P_{\eta_{i_k}^\perp}(u_k)/\sin\theta_k,e\right\rangle$. Using classical trigonometric identities we can write
$$
\left\langle v_{k+1}-v_k,e\right\rangle = 2\cos\left(\frac{f(\theta_k)+\theta_k}{2}\right)\gamma_k,
$$
where 
$$
\gamma_k=\cos\left(\frac{f(\theta_k)-\theta_k}{2}\right)\sin\phi+\sin\left(\frac{f(\theta_k)-\theta_k}{2}\right)h_k.
$$
To conclude the proof apply Lemma~\ref{lem:vkdiff}.
\end{proof}

\begin{theorem} \label{th:zigzag}
If $2\phi>\pi/2-f(\pi/2)$ then $Q$ has finite escaping time.
\end{theorem}

\begin{proof}
Let $v_0,\ldots,v_n$ be any sequence of consecutive reflection velocities. By Lemma~\ref{lem:vk}, 
\begin{equation}\label{eq:estimateL}
2\geq\left\langle v_{n}-v_0,e\right\rangle=\sum_{k=0}^{n-1}\|v_{k+1}-v_k\|\gamma_k\,,
\end{equation}
where
$$
\gamma_k=\cos\left(\frac{f(\theta_k)-\theta_k}{2}\right)\sin\phi+\sin\left(\frac{f(\theta_k)-\theta_k}{2}\right)h_k
$$
and $h_k=\left\langle P_{\eta_{i_k}^\perp}(u_k)/\sin\theta_k,e\right\rangle$. To estimate $\gamma_k$ from below note that $h_k\leq \cos\phi$. Thus
$$
\gamma_k\geq\sin\left(\phi +\frac{f(\theta_k)-\theta_k}{2}\right)\geq\sin\left(\phi+\frac{f(\pi/2)-\pi/2}{2}\right).
$$
By assumption $\mu:=\phi+\frac{f(\pi/2)-\pi/2}{2}>0$. Then, it follows from \eqref{eq:estimateL} that
$$
L(v_0,\ldots,v_n)<\frac{2}{\sin\mu}\,,
$$
for every sequence of consecutive reflection velocities $v_0,\ldots,v_n$. This proves that $Q$ has bounded zigzag reflections. Thus, by Lemma \ref{prop:escapingzigzag}, $Q$ has finite escaping time. 
\end{proof}

This theorem yields the following corollaries. 

\begin{corollary}\label{cor escaping 1}
Any polyhedral cone $Q$ with contracting reflection law $f$ sufficiently close to the specular one has finite escaping time.
\end{corollary}

\begin{proof}
It is clear that $2\phi>\pi/2-f(\pi/2)$ for every contraction $f$ sufficiently close to the identity. Thus, $Q$ has finite escaping time, by Theorem~\ref{th:zigzag}.
\end{proof}

Recall that a convex polyhedral cone $Q$ is obtuse if $\phi>\pi/4$.
\begin{corollary}\label{cor escaping 2}
Any obtuse polyhedral cone $Q$ has finite escaping time for every contracting reflection law $f$.
\end{corollary}

\begin{proof}
If the polyhedral cone is obtuse then $\phi>\pi/4$. Thus, $2\phi>\pi/2>\pi/2-f(\pi/2)$ for every contraction $f$. Thus, $Q$ has finite escaping time, by Theorem~\ref{th:zigzag}.
\end{proof}

\section{Uniform Expansion}
\label{sec:unifexp}

By Proposition~\ref{prop derivative in Jacobi coordinates},
the first component  of the derivative $ D \Phi_{f,P}(p,v)$ of the billiard map is represented in Jacobi coordinates by the map 
\begin{equation}
\label{Lvetav'}
L_{v,\eta,v'} := P_{{v'}^ \perp}\circ P_{v,\eta^ \perp}:\Rr^d\to\Rr^d
\end{equation}
 where $v',v,\eta\in\Rr^d$
  are three coplanar unit vectors   with $v'=C_\eta (R_\eta(v))$.

The main result of this section is Theorem~\ref{hyperb:main},
which gives conditions that ensure the uniform expansion of compositions of such maps.
Since the second component of the billiard map is contracting (see Proposition~\ref{prop:partialhyperbolic}),
these conditions will imply the uniform hyperbolicity of the  billiard map.

\subsection{Trajectories}

Let $P$ be a $d$-dimensional polytope in $\Rr^d$,
and ${\mathscr N}_P$ be the set of its unit inward normals. Denote by $\Nn_0$ the set of natural numbers $\Nn$ including $0$.

In the sequel we introduce a space of trajectories 
containing true orbits of the billiard map of $P$.
The reason is to exploit the compactness of this space
which does not hold for the billiard map's phase space,
since one has to exclude from the phase space all orbits which eventually hit the skeleton of $P$.

Define  the map
$h\colon D \to (\Ss\times\mathscr N_P)^{\Nn_0}$,
$h(p,v):=\{(v_j,\eta_{i_j}\}_{j\in \Nn_0}$ where for all $j\geq 0$,  $\Phi_{f,P}^j(p,v)=(p_j,v_j)$ with $p_j\in F_{i_j}$. Recall that $D$ is maximal invariant set defined in Section~\ref{section hyperbolicity}.
This map semi-conjugates the billiard map $\Phi_{f,P}$
with the  shift  on the space of sequences
$ (\Ss\times\mathscr N_P)^{\Nn_0}$.
Since $h(D)$ is not compact we introduce the following definition extending the notion of billiard trajectory.

 Although ${\mathscr N}_P=\{\eta_1,\ldots, \eta_N\}$, in order to simplify our notation  from now on  we will write $\eta_j$, $j\in\Nn_0$,  for any normal  in ${\mathscr N}_P$ and not necessarily the $j$-th normal in ${\mathscr N}_P$.

\begin{defn}\label{def trajectory}
A sequence $\{(v_j,\eta_j)\}_{j\geq0}\in (\Ss\times\mathscr N_P)^{\Nn_0}$ is called
a {\em trajectory} if for all $j\in\Nn$
\begin{enumerate}
\item $\langle v_{j-1},\eta_j\rangle \leq 0$,

\item $v_{j} = C_{\eta_{j}} \circ R_{\eta_{j}}(v_{j-1})$,
\end{enumerate}
where $R_\eta$ is the reflection introduced in section~\ref{sec:statements}, and  
 $C_\eta$ is the contracting reflection law defined in subsection~\ref{subsec: contractive laws}.
 We denote by $\mathscr{T}=\mathscr{T}_{f,P}$ the 
 space of all trajectories.
\end{defn}

Note that
$$ h(D)\subset \mathscr{T} \subset (\Ss\times \mathscr{N}_P)^ {\Nn_0}. $$

Given $i<j$ in $\Nn_0$, we denote by
$[i,j]:=\{i,i+1,\ldots, j\}\subseteq \Nn_0$ the time interval between the instants $i$ and $j$. Given a trajectory $\{(v_j,\eta_j)\}_{j\geq0}$ and a time interval $[i,j]$,  the linear span $V_{[i,j]}:=\linspan{ v_i,v_{i+1},\ldots, v_j}$ is called the
 {\em velocity front} of the trajectory along the time interval $[i,j]$. The linear span $N_{[i,j]}:=\linspan{ \eta_i,\eta_{i+1},\ldots, \eta_j}$ is called the \textit{normal front} of the trajectory along the time interval $[i,j]$. Given $i\in \Nn$, let $L_i:v_{i-1}^\perp\to v_{i}^{\perp}$ be the linear map defined by
$$
L_i=P_{{v_i}^ \perp}\circ P_{v_{i-1},\eta_i^ \perp}.
$$
Finally we define the 
{\em velocity tangent flow} along  $[i,j]$
to be the linear map  $L_{[i,j]}:v_i^\perp \to v_j^\perp$
defined by
$$L_{[i,j]} = 
L_{j}\circ \ldots \circ  L_{i+1}.$$
When the trajectory is associated to a billiard orbit $\{(p_l,v_l)\}_{l\geq0}$ of $\Phi_{f,P}$, the linear map $L_{[i,j]}$ represents, in Jacobi coordinates,
the first component  of the derivative  $D\Phi_{f,P}^{j-i}$ at $(p_i,v_i)$.
By definition, given $i<j<k$,
$$L_{[i,k]}= L_{[j,k]}\circ L_{[i,j]}\;.$$

We now extend Definition~\ref{def k-generating}
to trajectories.

\begin{defn}\label{def generating}
We say that the trajectory $\{(v_l,\eta_l)\}_{l\geq0}$ is {\em generating on $[i,j]$} if   $N_{[i,j]}=\Rr^d$. Given $k \in\Nn$, we say that the trajectory is {\em $k$-generating} if it is generating on any interval $[i,j]$ with  $j-i\geq k$.
\end{defn}

We can now state this section's main result.
\begin{theorem}\label{hyperb:main}
Given $\varepsilon>0$, $d$-dimensional polytope $P$ and contracting reflection law $f$, there exists a constant $\sigma=\sigma(\varepsilon,d,f)>1$ such that for any trajectory $\{(v_j,\eta_j)\}_{j\geq0}$  
in $\mathscr{T}_{f,P}$ the following holds. If
\begin{enumerate}
\item $P$ is $\varepsilon$-spanning,
\item $\{(v_j,\eta_j)\}_{j\geq0}$ is $k$-generating,
with $k\in\Nn$,
\end{enumerate}
then  the linear map $L_{[0,2 k]}:v_0^\perp \to v_{2 k}^\perp$ 
satisfies 
$$\norm{L_{[0,2 k]} (v)} \geq \sigma\,\norm{v}\quad \text{ for all
}\; v\in v_0^\perp.$$
\end{theorem}
The proof of this theorem is done at the end of the section.

\begin{remark}
From the previous theorem's conclusion, for any $n\geq 0$,
$$ \norm{L_{[0,n]}(v)}\geq \sigma^{\frac{n}{2 k}-1}\,\norm{v}  \quad \text{  all }\; v\in v_0^\perp\;. $$
This means, minimum growth expansion rate larger or equal than $\sigma^{\frac{1}{2 k}}>1$.
\end{remark}

\subsection{Properties of trajectories}
The following result says that the \textit{trajectory space} $\mathscr{T}$ is  compact. 

\begin{proposition} \label{traj:compact}
The space $\mathscr{T}$ is a closed subspace of the product space $(\Ss\times \mathscr{N}_P)^ {\Nn_0}$. In particular, with the induced topology
 $\mathscr{T}$ is a compact space.
\end{proposition}

\begin{proof} The trajectory space  $\mathscr{T}$ is closed in the product space 
because conditions (1) and (2) in Definition~\ref{def trajectory} are closed conditions. By Thychonoff's theorem  $(\Ss\times \mathscr{N}_P)^ {\Nn_0}$ is compact,
and hence $\mathscr{T}$ is compact too. 
\end{proof}

 \begin{lemma}\label{lem alphabeta}
Given any trajectory $\{(v_j,\eta_j)\}_{j\geq0}$ there exist scalars $\alpha_j,\beta_j\in\Rr$ such that for any $j\geq1$,
$$
v_j=\alpha_j \eta_j+\beta_j v_{j-1}
$$
where 
$$
\cos\left(\frac\pi2\lambda(f)\right)<\alpha_j<2\quad\text{and}\quad 0\leq \beta_j<1.
$$
Moreover, 
$$\left|\frac{\langle v_j,\eta_j\rangle}{\langle v_{j-1},\eta_j\rangle}\right|=\frac{\cos f(\theta_j)}{\cos\theta_j}\quad\text{where}\quad \theta_j=\arccos|\langle v_{j-1},\eta_j\rangle|.$$
\end{lemma}

\begin{proof}
According to Proposition \ref{prop contractive law},
$$
v_j = (a_j+b_j)\cos\theta_j\,\eta_j+b_j v_{j-1}
$$
where 
$$
a_j=\frac{\cos f(\theta_j)}{\cos\theta_j},\quad b_j=\frac{\sin f(\theta_j)}{\sin\theta_j}\quad\text{and}\quad\theta_j = \arccos|\langle v_{j-1},\eta_j\rangle|.
$$
Since $\lambda(f)<1$, we have $1\leq a_j+b_j<2$ and $0\leq b_j<1$. Moreover, $\cos\theta_j>\cos(\frac\pi2\lambda(f))$. The last claim is a simple computation.
\end{proof}

\begin{lemma} \label{vel:front} 
Given a trajectory $\{(v_l,\eta_l)\}_{l\geq0}$,
for all intervals $[i,j]$,
\begin{enumerate}
\item $V_{[i,j]} =
 \linspan{ v_i} + N_{[i+1,j]}$ and ${V_{[i,j]}}^\perp \subseteq v_i^\perp\cap v_j^\perp$.
\item $L_{[i,j]}:v_i^\perp \to v_j^\perp$ is the identity on ${V_{[i,j]}}^\perp$. 
\item $\minexp(L_{[i,j]})\geq 1$.
\end{enumerate}
\end{lemma}

\begin{proof}
Straightforward computation.
\end{proof}

\subsection{Collinearities}
Throughout the rest of this section, we assume that $\varepsilon>0$ is fixed and that $P$ is $\varepsilon$-spanning.

Consider a trajectory  $\{(v_l,\eta_l)\}_{l\geq 0}$ in
$\mathscr{T}$.

\begin{defn}
\label{def collinearity}
A time  interval $[i,j]$  is called
a {\em collinearity} of the trajectory $\{(v_l,\eta_l)\}_{l\geq0}$ if its velocity and the normal fronts along the time interval $[i,j]$ coincide, i.e. $V_{[i,j]} = N_{[i,j]}$. The number
$j-i$ will be referred as the length of the collinearity $[i,j]$.
\end{defn}

\begin{defn}
\label{def minimal collinearity}
A collinearity is called {\em minimal} if it contains no
smaller subinterval which is itself a collinearity.
\end{defn}
For instance, if $v_i\in\linspan{\eta_i}$ then $\{i\}$ is
a minimal collinearity of length $0$.

\begin{proposition} \label{colin:charact} 
Given a trajectory  $\{(v_l,\eta_l)\}_{l\geq0}$,  assume  $v_i\in N_{[i,j]}$ with $i\leq j$.
Then there is some $i'\in [i,j]$ such that
 the time interval $[i',j]$ is a collinearity.
\end{proposition}

\begin{proof}
The proof goes by induction on the length $r=j-i$.
If the length is $0$ then $i=j$ and we have necessarily  $v_i\in\linspan{\eta_i}$, in which case it is obvious that $[i,i]=\{i\}$
is a collinearity.
Assume now that the statement holds for all time intervals
of length less than $r$, and let 
$v_i=\lambda_i\eta_i+\cdots+ \lambda_j \eta_j$ with $j-i=r$.
We consider two cases:

First suppose that   $\lambda_i\neq 0$. By item  (1) of Lemma~\ref{vel:front},
$$
V_{[i,j]} = \linspan{v_i} + N_{[i+1,j]} \subseteq N_{[i,j]}.
$$
Conversely, because $\lambda_i\neq 0$ we have $\eta_i\in \linspan{v_i} + N_{[i+1,j]}$
which proves that
$$
N_{[i,j]} \subseteq \linspan{v_i} + N_{[i+1,j]} = V_{[i,j]},
$$
where in the last equality we have used again  item  (1) of Lemma~\ref{vel:front}. Therefore, $[i,j]$ is a collinearity in this case.

Assume next that $\lambda_i=  0$. By Lemma~\ref{lem alphabeta}, there are scalars $\alpha_{i+1}$ and $\beta_{i+1}$ such that $v_{i+1}=\alpha_{i+1}\eta_{i+1}+\beta_{i+1}v_i$. We may assume that $\beta_{i+1} \neq 0$. Otherwise $v_{i+1}\in\linspan{\eta_{i+1}}$ and $[i+1,j]$ is a collinearity. Thus
$$  \lambda_{i+1} \eta_{i+1} + \ldots + \lambda_j \eta_j = v_i = \frac{1}{\beta_{i+1}}\left(  v_{i+1} - \alpha_{i+1} \eta_{i+1}\right)\;. $$
In this case 
$$ v_{i+1}=\beta_{i+1}\,\left[ \, \left(\lambda_{i+1}-\frac{\alpha_{i+1}}{\beta_{i+1}}\right)\eta_{i+1} + \lambda_{i+2} \eta_{i+2} + \ldots + \lambda_j \eta_j\, \right] \;.$$
and the conclusion follows by the induction hypothesis applied to the time interval $[i+1,j]$ of length $p-1$.
\end{proof}

\begin{proposition} \label{colin:properties}Given a trajectory  $\{(v_l,\eta_l)\}_{l\geq0}$ and $i<j\leq j'$ the following holds:
\begin{enumerate}
\item If   $[i,j]$ is a collinearity then
$[i,j']$ is also a collinearity.
\item If  $v_j\in V_{[i,j-1]}$ and
$\eta_j\notin N_{[i,j-1]}$,
then there is some $i< i' \leq j$ such that 
$[i',j]$ is a collinearity.
\end{enumerate}
\end{proposition}

\begin{proof} 
Let $i<j\leq j'$.
\begin{enumerate}
\item Assume $V_{[i,j]} = N_{[i,j]}$. Then by Lemma~\ref{vel:front},
$$
V_{[i,j']} = \linspan{v_i} + N_{[i+1,j]} +  N_{[j+1,j']} = N_{[i,j]} +  N_{[j+1,j']} = N_{[i,j']}.
$$
\item Assume now $v_j\in V_{[i,j-1]}$. By Lemma~\ref{lem alphabeta}, 
$$
\eta_j = \frac{1}{\alpha_j}\left(v_j-\beta_j v_{j-1}\right),
$$
where $\alpha_j\neq 0$. Thus $\eta_j \in V_{[i,j-1]}$. By Lemma~\ref{vel:front} we can write $\eta_j = \lambda_i v_i + u$ for some $u\in N_{[i+1,j-1]}$. By assumption, $\lambda_i\neq 0$. Thus $v_i \in N_{[i+1,j]}$. Again by Lemma~\ref{lem alphabeta}, we conclude that $v_{i+1}\in N_{[i+1,j]}$. Now the claim follows by Proposition~\ref{colin:charact}.
\end{enumerate} 
\end{proof}

\begin{corollary}\label{dichotomy}
Let $\{(v_l,\eta_l)\}_{l\geq 0}$ be a trajectory  
and   $k\geq i\geq 0$ be integers such that the time segment
 $[i,k]$  contains no subinterval which is a collinearity. 
Then for
every $j \in [i,k]$  either
\begin{enumerate}
\item  $\eta_j\in \{\eta_{i+1},\ldots,\eta_{j-1}\}$, or else 
\item  $v_j\notin V_{[i,j-1]}$.
\end{enumerate}
\end{corollary}

\begin{proof}
This corollary is a reformulation of item (2) of Proposition~ \ref{colin:properties}.
\end{proof}
 
\subsection{Quantifying collinearities}

We are now going to prove quantified versions of 
Propositions~\ref{colin:charact},~\ref{colin:properties} and Corollary~\ref{dichotomy}. The following abstract continuity lemma will be useful.

\begin{lemma}\label{abstr:cont}
Let $\Xscr$ be a compact topological space and $f,g:\Xscr\to\Rr$ be continuous functions
such that $g(x)=0$ for all $x\in \Xscr$ with  $f(x)=0$.
Given $\delta>0$ there is $\delta'>0$ such that for all $x\in \Xscr$,
if $f(x)<\delta'$ then $g(x)<\delta$.
\end{lemma}

\begin{proof}
Assume, to get a contradiction, that the claimed statement does not hold.
Then there is $\delta>0$ such that for all $n\in\Nn$ there is a point $x_n\in \Xscr$
with $f(x_n)<\frac{1}{n}$ and $g(x_n)\geq\delta$. Since $\Xscr$ is compact, by taking a subsequence
we can assume $x_n\to x$ in $\Xscr$. By continuity of $f$ and $g$,
$f(x)=0$ and $g(x)\geq \delta$, which contradicts the lemma hypothesis.
\end{proof}

\begin{defn}
\label{def delta collinearity}
Given $\delta>0$, we call  {\em $\delta$-collinearity} of a  trajectory  $\{(v_l,\eta_l)\}_l$
to any time  interval $[i,j]$ such that $\dim{V_{[i,j]}}=\dim{N_{[i,j]}}$  and 
$$\angle\left( V_{[i,j]}, N_{[i,j]} \right) <\delta \;.$$
\end{defn}

\begin{proposition}\label{delta:colin:charact}
Given $\delta>0$  there exists $\delta'>0$ such that for any trajectory  $\{(v_l,\eta_l)\}_l$ the following holds. If $$\angle\left( v_i, N_{[i,j]} \right)<\delta'$$ for some $0\leq i\leq j$, then there exists $i'\in [i,j]$ for which the time interval $[i',j]$ is a $\delta$-collinearity of the given trajectory.
\end{proposition}

\begin{proof}
Notice that, because the space of trajectories $\mathscr{T}$ is shift invariant,
there is no loss of generality in assuming that $[i,j]=[0,p]$. 
For each $k\geq 0$, define the functions
$f_{k},g_{k}:\mathscr{T}\to\Rr$ by
\begin{align*}
f_k\left( \{(v_l,\eta_l)\}_{l} \right) &= \angle\left( v_0, N_{[0,k]} \right),\\ g_k\left(\{(v_l,\eta_l)\}_{l} \right)&=  \min_{0\leq i \leq k}\angle\left( V_{[i,k]}, N_{[i,k]} \right).
\end{align*}
These functions are clearly continuous.

Proposition~\ref{colin:charact}
shows that for all $x\in \mathscr{T}$ and $0\leq  k \leq p$,
$f_k\left(x \right)=0$ implies
$g_k\left(x\right)=0$.
Thus, given $\delta>0$, by Lemma~\ref{abstr:cont}, there exists $\delta'>0$ such that for any $0\leq k \leq p$ and $x\in \mathscr{T}$, 
\begin{equation*}
f_k\left(x \right)<\delta'\quad \Rightarrow\quad 
 g_k\left(x\right)<\delta\;.
\end{equation*}
\end{proof}

\begin{proposition}\label{delta:colin:properties} Given any trajectory  $\{(v_l,\eta_l)\}_{l}$, $i<j\leq j'$ and $\delta>0$ the following holds.
\begin{enumerate}
\item \label{delta item 1}  If $[i,j]$ is a $\delta$-collinearity, then
$[i,j']$ is  a $\delta'$-collinearity, where $\delta'=\arcsin(\frac{\sin\delta}{\sin\varepsilon} )$.
\item \label{delta item 2}  There exists $\delta'>0$ such that, if  
$$\angle(v_j, V_{[i,j-1]})<\delta'$$
 and $\eta_j\notin N_{[i,j-1]}$,
then there is some $i < i' \leq j$ such that 
$[i',j]$ is a $\delta$-collinearity.
\end{enumerate}
\end{proposition}

\begin{proof} 
Denote by $H$ the linear space spanned by the `new' normals
$\eta_l$ in the range $j<l\leq j'$, i.e., normals  which are not in $\{\eta_i,\ldots, \eta_j\}$.
 By definition of $H$ we have,
\begin{align*}
V_{[i,j']}  &= V_{[i,j]} + H \;, \\
N_{[i,j']} &= N_{[i,j]} + H \;. 
\end{align*}
Hence by Lemma~\ref{angle:lemma}, if $[i,j]$ is a $\delta$-collinearity,
\begin{align*}
 \sin  \angle\left(V_{[i,j']},  N_{[i,j']} \right)
& \leq \frac{1}{\sin\varepsilon} \sin  \angle\left(V_{[i,j]},  N_{[i,j]} \right)  \leq \frac{\sin\delta}{\sin\varepsilon} = \sin\delta' ,
\end{align*}
which proves that $[i,j']$ is a $\delta'$-collinearity.
This proves~\eqref{delta item 1}.

To prove item~\eqref{delta item 2} note that, as in the proof of Proposition~\ref{delta:colin:charact},
there is no loss of generality in assuming that $[i,j]=[0,p]$. 
Define the functions
$f_{k},g_{k}:\mathscr{T}\to\Rr$ by
\begin{align*}
f_k\left( \{(v_l,\eta_l)\}_{l} \right) &= \angle\left( v_k, V_{[0,k-1]} \right),\\ g_k\left(\{(v_l,\eta_l)\}_{l} \right)&=  \min_{1\leq i \leq k}\angle\left( V_{[i,k]}, N_{[i,k]} \right).
\end{align*}
These functions are clearly continuous.

Item (2) of Proposition~\ref{colin:properties}
shows that for every $x=\{(v_l,\eta_l)_l\}\in \mathscr{T}$ and for every $1\leq k \leq p$ for which $\eta_k\notin N_{[0,k-1]}$, $f_k\left(x \right)=0$ implies
$g_k\left(x\right)=0$.
Thus, given $\delta>0$, by Lemma~\ref{abstr:cont}, there exists $\delta'>0$ such that for every $x=\{(v_l,\eta_l)_l\}\in \mathscr{T}$ and for every $1\leq  k \leq p$ for which $\eta_k\notin N_{[0,k-1]}$, 
\begin{equation*}
f_k\left(x \right)<\delta'\quad \Rightarrow\quad 
 g_k\left(x\right)<\delta\;.
\end{equation*}
This proves~\eqref{delta item 2}.
\end{proof}

\begin{corollary}\label{delta:dichotomy}
Given $\delta>0$ there is $\delta'>0$ such that the following dichotomy holds.
Let $[i+1,j]$ be a time segment of a trajectory that contains no subinterval which is a $\delta$-collinearity
of that trajectory. Then for
every $l \in [i+1,j]$  either
\begin{enumerate}
\item  $\eta_l\in \{\eta_{i+1},\ldots,\eta_{l-1}\}$, or else 
\item  $\angle( v_l, V_{[i,l-1]} )\geq \delta'$.
\end{enumerate}
\end{corollary}

\begin{proof}
This corollary is a reformulation of Proposition~ \ref{delta:colin:properties}\,(2).
\end{proof}

\subsection{Expansivity lemmas} 
Recall the map $L_{v,\eta,v'}$ defined in~\eqref{Lvetav'}.
The first lemma says 
that this map has two
singular values: $\lambda=1$ with multiplicity $d-1$,
and $\lambda= \abs{\langle v',\eta\rangle/ \langle v,\eta\rangle}$ with multiplicity $1$. See Figure~\ref{proj:comp}.

 \begin{figure}
  \begin{center}
    \includegraphics*[width=3in]{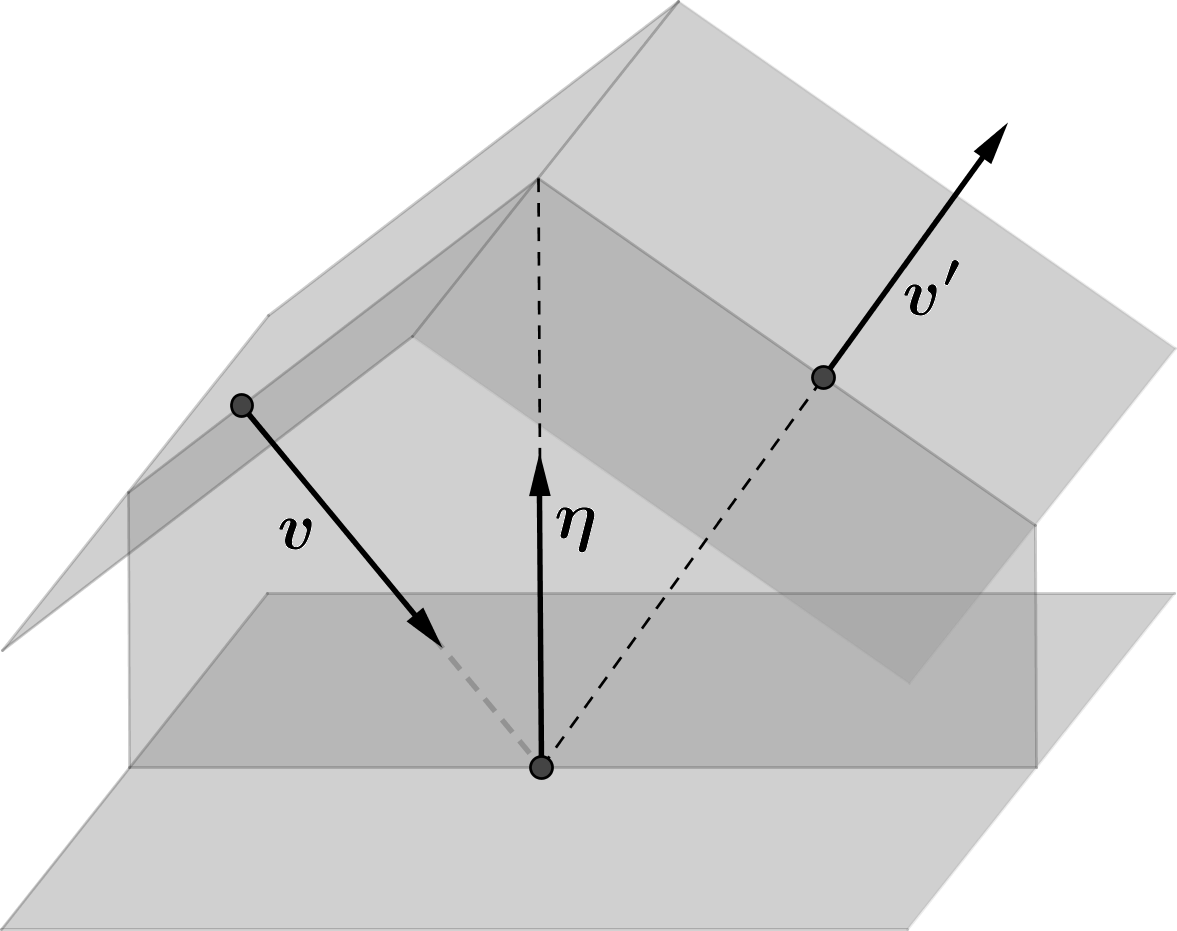}
  \end{center}
  \caption{Composition of the projections $P_{{v'}^ \perp}\circ P_{v,\eta^ \perp}$ }
  \label{proj:comp}
\end{figure}

\begin{lemma}\label{proj:sing:val}
Given coplanar unit vectors $v',v,\eta\in\Rr^d$, 
the composition
$L_{v,\eta,v'} :\Rr^d\to\Rr^d$ satisfies:
\begin{enumerate}
\item[(a)] $L_{v,\eta,v'}(v)=0$,
\item[(b)] $L_{v,\eta,v'}(x)=x$, for every $x\in \eta^ \perp\cap v^ \perp$,
\item[(c)] $L_{v,\eta,v'}$ maps the line $v^\perp\cap W$
onto the line $v'^\perp\cap W$, where $W=\linspan{v,\eta}$,  multiplying the vector's norms by the factor
$\abs{\langle v',\eta\rangle/ \langle v,\eta\rangle}$.
\end{enumerate}
\end{lemma}
 
\begin{proof}
Straightforward computation.
\end{proof}

\begin{remark}
\label{proj:sing:val id}
If $v,\eta,v'$ are collinear vectors then
$L_{v,\eta,v'}=\id$ on $v^\perp$.
\end{remark}
  
\medskip

The remaining lemmas are abstract. Let $V$, $V'$, $V''$ be Euclidean spaces of the same
dimension, and $L:V\to V'$,  $L':V'\to V''$ be linear
isomorphisms.
  
Given $\sigma\geq1$ and a subspace $E\subset V$, we say that $L$ is a {\em $\sigma$-expansion on $E$} if   $\norm{L v}\geq \sigma \norm{v}$ for all $v\in E$, i.e., $\minexp(L\vert_E)\geq \sigma$.
Given another linear subspace $H\subseteq V$ such that $E\subseteq H$ we say that $L$ is a {\em relative $\sigma$-expansion on $H$ w.r.t. $E$} if and only if the quotient map  
  $\overline{L}: {V/E}\to  {V'/L(E)}$ is a $\sigma$-expansion on ${H/E}$. Note that the quotient space $V/E$ is an Euclidean space which can naturally be identified with $E^\perp$.
Finally, we say that $L$ is a {\em $\sigma$-expansion} to mean that $L$ is a $\sigma$-expansion on its domain $V$.

If we  do not need to specify the minimal rate of expansion we shall simply say that $L$ is a uniform expansion on $E$, or that $L$ is a relative uniform expansion on $H$ w.r.t. $E$.

\begin{lemma}\label{abstr:expans:lemma}
Given a linear subspace $H\subseteq V$,
if 
\begin{enumerate}
\item $L$ is a $\sigma$-expansion on $H$, and
\item $L$ is  a relative $\sigma$-expansion on $V$ w.r.t. $H$
\end{enumerate} 
 then $L$ is  a $\sigma$-expansion on $V$.
\end{lemma}

\begin{proof}
Follows immediately from the definition of $\sigma$-expansion and relative $\sigma$-expansion.
\end{proof}

  We will now derive some explicit formulas for the minimum expansion of compositions of  linear expanding  maps. For that purpose we
introduce an exotic operation on the set $[0,1]$ 
that plays a key role in these formulas.
$$ a\oplus b := a+b -a\,b\;.$$
With it, 
$([0,1],\oplus)$ becomes a commutative semigroup isomorphic to $([0,1],\cdot)$. In fact, the map $\Psi:([0,1],\oplus)\to ([0,1],\cdot)$, $\Psi(x):= 1-x$,
is a semigroup isomorphism. 
The numbers $0$ and $1$ are respectively the neutral and the absorbing elements of the semigroup $([0,1],\oplus)$.
 An important property of this exotic operation is that for all $a,b\in [0,1]$,
$$ a\oplus b <1 \; \Leftrightarrow\; a<1 \; \text{ and }\; b<1 .$$
We will write  
$\oplus_n x := x\oplus \ldots \oplus x$ for any $\oplus$-sum of  $n$ equal terms $x\in [0,1]$.

The following lemmas use the notation introduced in subsection~\ref{Euclidean Geometry}. 
 
\begin{lemma}\label{exp:2}
Let $L,L':\Rr^ 2\to\Rr^2$ be  linear maps 
such that $\minexp(L)\geq 1$, $\minexp(L')\geq 1$,
  $\lambda=\norm{L}>1$ and $\lambda'=\norm{L'}>1$.
If the sine of the angle between $\med(L^\ast)$ and $\med(L')$
is at least $\varepsilon$, then $L'\circ L$ has minimum expansion
$$\minexp (L'\circ L) \geq \frac{1}{\sqrt{ (1-\varepsilon^2)\oplus \lambda^{-2}\oplus \lambda'^{-2} }
} >1\;.$$
\end{lemma}

\begin{proof}
Without loss of generality we can assume that the automorphisms $L,L'\in\GL(\Rr^2)$ have two singular values, respectively  $\{1,\lambda\}$ with $1<\lambda$, and
$\{1,\lambda'\}$ with $1<\lambda'$.
Otherwise simply normalize $L$ and $L'$ dividing them by the minimum expansion.
Hence $\norm{L^{-1}}=1=\norm{(L')^{-1}}$.
These maps have  gap ratios $\norm{L^{-1}}/\minexp(L^{-1})=\lambda$
and 
$\norm{(L')^{-1}}/\minexp((L')^{-1})=\lambda'$.
The conclusion of this lemma will folllow  from~\cite[Proposition 2.23]{DK-book}
applied to the composition of linear maps
$L^{-1}\circ (L')^{-1}$. The quantity
$\alpha((L')^{-1},L^{-1})$ in that proposition is
the cosine of the angle between the most expanding directions of the linear maps
$((L')^{-1})^\ast$ and $L^{-1}$ which coincide with the least expanding directions of
$L'$ and $L^\ast$, respectively. Since these directions are orthogonal to 
$\med(L')$ and $\med(L^\ast)$ we have
$$\alpha((L')^{-1},L^{-1})^2 = \cos^2 \angle(  \med(L'), \med(L^\ast))\leq 1-\varepsilon^2 .$$
Thus by~\cite[Proposition 2.23]{DK-book} 
\begin{align*}
\norm{(L'\circ L)^{-1}} &= \norm{ L^{-1}\circ (L')^{-1}} = \frac{ \norm{ L^{-1}\circ (L')^{-1}} }{ \norm{ L^{-1}}\,\norm{ (L')^{-1}} }\\
& \leq 
\beta( (L')^{-1},  L^{-1} )  
 \leq   \sqrt{ (1-\varepsilon^2)\oplus \lambda^{-2}\oplus \lambda'^{-2} } .
\end{align*}
Since $\minexp(L'\circ L)=1/\norm{(L'\circ L)^{-1}}$,
the claim follows.
\end{proof}

\begin{lemma} \label{expanding:comp}
Consider linear maps
$L\colon V\to V'$ e $L'\colon V'\to V''$ between Euclidean spaces of dimension $d$.
Given $1\leq k <d$ assume that
\begin{enumerate}
\item $\minexp(L)\geq 1$ and $\lambda=s_k(L) >1$,
 
\item   $\lambda' = \norm{L'}>1=s_2(L')$.
\end{enumerate}
If \,  $\sin \angle( \med(L'),\med^{\geq}_\lambda(L^\ast))\geq \varepsilon$ \, then  
$$s_{k+1}(L'\circ L)\geq 
\frac{1}{\sqrt{ (1-\varepsilon^2)\oplus \lambda^{-2}\oplus \lambda'^{-2} } }\;> 1. $$
\end{lemma}

\begin{proof} 
We can assume that $V=V'=V''=\Rr^d$.
Consider the singular value decomposition $L=U\,D\,V$,
where $U$ and $V$ are orthogonal matrices,  and $D=(D_{ij})$ is the diagonal matrix with
diagonal entries $D_{ii}=s_i(L)$ for $i=1,\ldots, d$. We can factor $D$ as a product $D=\hat{D}\,D_0$
of two diagonal matrices: $\hat{D}=\left[ \begin{array}{cc}
\lambda I_k & 0 \\ 0 & I_{d-k} 
\end{array}\right]$ and $D_0$ with  diagonal entries 
$D^{(0)}_{ii}= {D_{ii}}/{\hat{D}_{ii}}\geq 1$.
Set $\hat{L}=U\,\hat{D}$ and $L_0=D_0\,V$, so that $L=\hat{L}\circ L_0$.
The linear mapping $\hat{L}$ has singular values 
$$s_1(\hat{L})=\ldots = s_k(\hat{L})=\lambda>1 =s_{k+1}(\hat{L})=\ldots = s_d(\hat{L}),$$
while $\minexp(L_0)\geq 1$.
Hence $s_{k+1}(L'\circ L)\geq s_{k+1}(L'\circ \hat{L})$. To simplify the geometry we assume from now on that
$L=\hat{L}$.
 
Take a unit vector  $v'\in \med(L')$. By assumption
$v'\notin \med^\geq_\lambda(L^\ast)$. Let $v_0'$ denote the orthogonal projection of $v'$ onto $\med^\geq_\lambda(L^\ast)$.
These two vectors span a plane $P_0:=\linspan{v',v_0'}$.
Define also the subspaces  
$E_0:= \med^\geq_\lambda(L^\ast)\cap (v_0')^\perp$ and $G_0:= \med^\geq_\lambda(L^\ast)^\perp\cap (v')^\perp$.
These three subspaces determine an orthogonal decomposition $\Rr^d=P_0\oplus E_0\oplus G_0$.
Because the mappings $L:\med^\geq_\lambda(L)\to \med^\geq_\lambda(L^\ast)$ and
$L:\med^\geq_\lambda(L)^\perp\to \med^\geq_\lambda(L^\ast)^\perp$ are both conformal,
it follows that $\Rr^d=P_-\oplus E_-\oplus G_-$ is an orthogonal decomposition,
where $P_-:=L^{-1} P_0$, $E_-:=L^{-1} E_0$ and $G_-:=L^{-1} G_0$.
In fact, since  $P_0\cap \med^\geq_\lambda(L^\ast) \perp E_0$
and $L:\med^\geq_\lambda(L)\to \med^\geq_\lambda(L^\ast)$ is conformal their pre-images are also
orthogonal,  $P_-\cap \med^\geq_\lambda(L) \perp E_-$.
Similarly, since  $ \med^\geq_\lambda(L^\ast)^\perp \cap P_0\perp G_0$
and $L:\med^\geq_\lambda(L)^\perp\to \med^\geq_\lambda(L^\ast)^\perp$ is conformal their pre-images are also
orthogonal,  $\med^\geq_\lambda(L)^\perp \cap P_-\perp G_-$.
Define now $P_+:=L' P_0$, $E_+:=L' E_0$ and $G_+:=L' G_0$.
Because $\minexp(L')\geq 1$   with $\med(L')\subset P_0$ and $\med((L')^\ast)\subset P_+$, it follows that
$\Rr^d=P_+\oplus E_+\oplus G_+$ is also an orthogonal decomposition.
Therefore the singular values of $L'\circ L$ are the singular values of the
restricted compositions $L'\vert_{P_0}\circ L\vert_{P_-}$,
$L'\vert_{E_0}\circ L\vert_{E_-}$ and $L'\vert_{G_0}\circ L\vert_{G_-}$.
Applying Lemma~\ref{exp:2} to the linear maps
 $L'\vert_{P_0}$ and $L\vert_{P_-}$ we see that
  $L'\vert_{P_0}\circ L\vert_{P_-}$ has minimum expansion
 $\beta:=\left( (1-\varepsilon^2)\oplus \lambda^{-2}\oplus \lambda'^{-2}  \right)^{-1/2}$.
The map $L'\vert_{E_0}:E_0\to E_+$ is an isometry while
$L\vert_{E_-}:E_-\to E_0$ is $\lambda$-conformal.
Therefore the second composition has a unique singular value $\lambda$
with multiplicity $k-1=\dim E_0$.
Note that $(1-\varepsilon^2)\oplus \lambda^{-2}\oplus \lambda'^{-2} \geq \lambda^{-2}$ which implies that
 $\lambda \geq \beta$.
Finally notice  that  $L\vert_{G_-}:G_-\to G_0$ and $L'\vert_{G_0}:G_0\to G_+$ are isometries.
Hence $1$ is the only the singular value of the third composition.
Since $\dim(P_-\oplus E_-)=\dim(P_0\oplus E_0)=k+1$, this proves that   $s_{k+1}(L'\circ L)\geq \min\{\beta,\lambda\}=\beta$.
\end{proof}

The next lemma is designed to be applied 
to a sequence of linear maps $L_{v_{i-1},\eta_i,v_{i}}\colon v_{i-1}^\perp\to v_i^\perp$ associated with an orbit segment of the billiard map $\Phi_{f,P}$. Compare assumptions (1)-(2) of Lemma~\ref{main:lin:alg} with the conclusions of Lemma~\ref{proj:sing:val} and Remark~\ref{proj:sing:val id}.

\begin{lemma} \label{main:lin:alg}
Given $\varepsilon>0$ and $\lambda>1$ consider unit vectors $\{v_0,v_1,\ldots, v_n\}$ in $\Rr^d$ and a family of linear maps
$L_i\colon v_{i-1}^\perp\to v_i^\perp$,
$1\leq i\leq n$,  such that:
\begin{enumerate}
\item if $u\in v_{i-1}^\perp\cap v_i^\perp$ then $L_i\, u=u$,
\item if   $\{v_{i-1},v_i\}$ are linearly independent and $u\in  v_{i-1}^\perp\cap \linspan{v_{i-1},v_i}$ then 
$L_i u\in  v_{i}^\perp \cap \linspan{v_{i-1},v_i}$ and $\norm{L_i\, u} \geq  \lambda \,\norm{u}$.

\item  for each $i=1,\ldots, n$, either $v_{i}\in\linspan{v_0,\ldots, v_{i-1}}$ or else
$$\sin \angle\left( v_{i}, \linspan{v_0,\ldots, v_{i-1}}\right) \geq \varepsilon. $$
\end{enumerate}

Then the composition map $L^{(n)}:=L_n\circ\ldots \circ L_0:v_0^\perp\to v_n^\perp$ satisfies 
$\norm{L^{(n)}\, u}\geq \sigma\,\norm{u}$ for all $u\in v_0^\perp \cap \linspan{v_0,\ldots, v_n}$,
where
 $$\sigma=1/\sqrt{ (\oplus_{d-1} (1-\varepsilon^2))\oplus (\oplus_{d} \lambda^{-2}) }>1.$$
\end{lemma}

\begin{proof} For each $i=1,\ldots, n$ define,
 $L^{(i)}:= L_i\circ\ldots \circ L_0:v_0^\perp \to v_i^\perp$ and $W_i:=\linspan{v_0,\ldots, v_i}$.
Since  $W_i^\perp = v_0^\perp \cap \ldots \cap v_i^\perp \subset  \cap_{j=0}^i v_j^\perp$,
and every vector $u\in W_i^ \perp$
 is fixed by all $L_j$ with $0\leq j\leq  i$,
 we have  $L^{(i)}\,u=u$ for every $u\in W_i^\perp$.  
 
 We can delete from  $\{v_0,v_1,\ldots, v_n\}$
all vectors $v_i$ such that $\{v_{i-1},v_i\}$
is linearly dependent, which by item (1) correspond to  maps $L_i=\id$, and in this way assume that for all $i=1,\ldots, n$, the vectors $\{v_{i-1},v_i\}$
are linearly independent and $\norm{L_i}\geq \lambda$.

 Because
 $W_i^\perp$ is a singular subspace it follows that  $L^{(i)}(v_0^\perp\cap W_i)=v_i^\perp\cap W_i$.
We claim that  $L^{(i)}:v_0^\perp\cap W_i \to v_i^\perp\cap W_i$ is a $\sigma_i$-expansion, where
$k_i:= \dim(W_i)-1$ and
\begin{align*}
\sigma_i &:= 1/\sqrt{(\oplus_{{k_i}-1}(1-\varepsilon^2)\oplus (\oplus_{k_i} \lambda^{-2}) } .
\end{align*} 
The proof of this claim goes by induction in $i$, applying Lemma~\ref{expanding:comp}.

The claim  holds for $i=1$ with $k_1=1$ and $\sigma_1=\lambda$.

Assume now (induction hypothesis) that
$L^{(i-1)}$ is a $\sigma_{i-1}$-expansion on $v_0^\perp\cap W_{i-1}$.
We know by assumption (3) that
either $v_i\in W_{i-1}$ or else
$\sin \left( \angle (v_i, W_{i-1}) \right) \geq\varepsilon$. 

\medskip

\blob\;  Assume first that $\sin \left( \angle (v_i, W_{i-1}) \right) \geq\varepsilon$.

We have
$\med^{\geq}_{\sigma_{i-1}}((L^{(i-1)})^\ast)=v_{i-1}^\perp\cap W_{i-1}$
and
$\med(L_i)=v_{i-1}^\perp\cap \linspan{v_{i-1},v_i}$.
To apply Proposition~\ref{expanding:comp} we need to check that
\begin{equation}
\label{cond2check} 
\sin \angle\left(v_{i-1}^\perp\cap W_{i-1}, \, v_{i-1}^\perp\cap \linspan{v_{i-1},v_i} \right) \geq\varepsilon.
\end{equation}

Let $v_i^0$ denote the unit vector obtained normalizing the orthogonal projection
of $v_i$ onto $v_{i-1}^\perp$, so that
$$ v_i = (\cos\alpha)\, v_i^0 +(\sin \alpha)\, v_{i-1}\;,$$
with $\norm{v_i^0} =1$, $\langle v_i^0, v_{i-1}\rangle = 0$ 
and where $\alpha=\angle(v_i^0,v_i)$.
Note that $v_{i-1}^\perp\cap \linspan{v_{i-1},v_i}$ is the line spanned by $v_i^0$.
Take any unit vector $v\in v_{i-1}^\perp\cap W_{i-1}$
and let us prove that $\sin (\angle(v_i^0,v))\geq \varepsilon$.
This will establish~\eqref{cond2check}. 
Define
$$ v':= (\cos\alpha)\, v +(\sin \alpha)\, v_{i-1} $$
which is  a unit vector in $W_{i-1}$.
We can assume that $\langle v_i^0, v\rangle \geq 0$ for otherwise
the angle $\angle ( v_i^0,v)$ that we want to minimize would be obtuse.
Using the previous expressions for $v_i$ and $v'$ we have
$$ \langle v_i, v'\rangle =
(\cos^2\alpha) \langle v_i^0, v\rangle +\sin^2\alpha \;.$$
Since this expresses $\langle v_i, v'\rangle$ as a convex combination between $\langle v_i^0, v\rangle$ and the number $1$, it follows that
$$\cos (\angle(v_i, v')) = \langle v_i, v'\rangle \geq \langle v_i^0, v\rangle = \cos (\angle(v_i^0, v))  $$
which implies that
$$ \sin \left( \angle (v_i^0, v) \right) \geq 
\sin \left( \angle (v_i, v') \right) \geq \varepsilon\;.$$
 This proves~\eqref{cond2check} and shows the assumptions of 
Proposition~\ref{expanding:comp} are met. From this proposition,
we get that  on the linear subspace $v_{0}^\perp\cap W_i$,  of dimension $k_i=k_{i-1}+1$, the linear map 
$L^{(i)}=L_i\circ L^{(i-1)}$ is 
a $\hat{\sigma}_i$-expansion 
where  
 $$ \hat{\sigma}_i:= \left((1-\varepsilon^2)\oplus \sigma_{i-1}^{-2}\oplus \lambda^{-2}\right)^{-1/2} \geq \sigma_i .  $$

\blob\; The case $v_i\in W_{i-1}$ is somewhat simpler. We have $W_i=W_{i-1}$,
$k_i=k_{i-1}$, and  $\sigma_i=\sigma_{i-1}$. Hence, since $\minexp(L_i)\geq 1$ by induction hypothesis the map $L^{(i)}=L_i\circ L^{(i-1)}$ is a $\sigma_i$-expansion on
$v_0^\perp\cap W_i$.
\end{proof}

\bigskip

\subsection{Proof of Theorem~\ref{hyperb:main}}
In this subsection we relate collinearities with expansion of the velocity tangent flow, and then prove Theorem~\ref{hyperb:main}.

Recall that we are assuming that $P$ is $\varepsilon$-spanning.

\begin{proposition}\label{rel:expans}
There exists $\sigma>1$, depending only on $d$, $f$ and  $\varepsilon$, such that given a collinearity $[i,j_0]$ of some trajectory, for all $j>j_0$,
the velocity flow $L_{[i,j]}$ is a relative $\sigma$-expansion on  
$ v_i^\perp\cap V_{[i,j]}$
w.r.t.  $ v_i^\perp\cap V_{[i,j_0]}$.
\end{proposition}

\begin{proof} Assume $\{(v_l,\eta_l)\}_l$ is
a trajectory with collinearity $[i,j_0]$.
Because $P$ is $\varepsilon$-spanning, for all $j>j_0$ such that
 $\eta_j\notin N_{[i,j-1]}$
we have $\angle(\eta_j, N_{[i,j-1]} )\geq \varepsilon$. 

Notice that 
$V_{[i,j-1]} = N_{[i,j-1]}$, for all $j>j_0$,
and by Lemma~\ref{lem alphabeta}, we have
$v_j=\alpha_j \eta_j + \beta_j v_{j-1}$ with $\alpha_j\geq \cos(\frac{\pi}{2}\lambda(f))>0$.
Hence  there is some
 $0<\varepsilon'<\varepsilon$ depending on $\varepsilon$ and on $\lambda(f)$, such that for all $j>j_0$ with  $\eta_j\notin N_{[i,j-1]}$,
$$ \angle(v_j, V_{[i,j-1]} )\geq \varepsilon'\;. $$

Consider the set of `new normal' times
$$ J:=\{ j_0 <l \leq j \colon \eta_l\notin N_{[i,l-1]} \} $$
and the corresponding velocity subspace
$$ V_J:= \linspan{ v_l \colon l\in J } , $$
so that $V_{[i,j]}= V_{[i,j_0]}\oplus V_J $.

By Lemma~\ref{main:lin:alg} there exists $\sigma>1$, depending only on $d$, $f$ and  $\varepsilon$  such that
$L_{[j_0,j]}$ is a 
$\sigma$-expansion on $v_{j_0}^\perp\cap V_{[j_0,j]}$. In particular,
$L_{[j_0,j]}$ is also a $\sigma$-expansion on
$v_{j_0}^\perp\cap V_J\subseteq v_{j_0}^\perp\cap V_{[j_0,j]}$.
By Lemma~\ref{anglo cumulativo} we have
$$\angle_{\min}(V_{[i,j_0]},V_J)\geq \arcsin( \sin^d(\varepsilon')) =: \tilde{\varepsilon}.$$
Hence there exists $1<\tilde{\sigma}<\sigma$
depending only on $\tilde{\varepsilon}$ and $\sigma$ such that
$L_{[j_0,j]}$ is a 
$\tilde{\sigma}$-expansion on $  (V_{[i,j_0]})^\perp \cap V_{[j_0,j]}$.
This implies that $L_{[i,j]}$ is a relative $\tilde{\sigma}$-expansion on  
$ v_i^\perp\cap V_{[i,j]}$
w.r.t.  $ v_i^\perp\cap V_{[i,j_0]}$.
\end{proof}

\begin{corollary}\label{delta:rel:expans}
Given the constant $\sigma>1$ in Proposition~ \ref{rel:expans}, and $1<\sigma'<\sigma$, there is $\delta>0$ such that for every
trajectory $\{(v_l,\eta_l)\}_{l}$, if $[i,j_0]$ is a $\delta$-collinearity  then for all $j>j_0$,
the velocity flow $L_{[i,j]}$ is a relative $\sigma'$-expansion on  
$ v_i^\perp\cap V_{[i,j]}$
w.r.t.  $ v_i^\perp\cap V_{[i,j_0]}$.

\end{corollary}

\begin{proof} This follows from Proposition~\ref{rel:expans}
with a  continuity argument like the one used in
the proof of Proposition~\ref{delta:colin:charact}.
\end{proof}

\bigskip

\begin{proposition}\label{expans} Given $\delta>0$ there exists $\sigma>1$, depending on  $d$, $f$, $\varepsilon$ and $\delta$, such that if a time interval $[i+1,j]$ of some trajectory contains no subinterval which is
a $\delta$-collinearity then $L_{[i,j]}$ is a $\sigma$-expansion on  
$ v_i^\perp\cap V_{[i,j]}$.
\end{proposition}

\begin{proof}
Let $[i,j]$ be a time interval such that $[i+1,j]$ contains no subinterval which is itself a $\delta$-collinearity.
By Corollary~\ref{delta:dichotomy}, there is $\delta'>0$ such that  for
every $l \in [i+1,j]$  either $\eta_l\in \{\eta_{i+1},\ldots,\eta_{l-1}\}$, or else 
$$\angle( v_l, V_{[i,l-1]})\geq \delta'\;.$$
Thus by Lemma~\ref{main:lin:alg}  $L_{[i,j]}$ is a uniform expansion on  
$ v_i^\perp\cap V_{[i,j]}$.
\end{proof}

\bigskip

Now we can prove Theorem~\ref{hyperb:main}.

\begin{proof}[Proof of Theorem~\ref{hyperb:main}]
Take the constant $\sigma>1$  given in Proposition~\ref{rel:expans}. Set $\sigma'=\frac{1}{2}+\frac{1}{2}\sigma\in (1,\sigma)$,
and pick $\delta=\delta(\sigma')>0$ as provided by Corollary~\ref{delta:rel:expans}. 
Fix the constant $\sigma''=\sigma(\delta)>1$ given by Proposition~\ref{expans} and set $\sigma_0=\min\{\sigma',\sigma''\}$.

Fix some integer $k\geq 0$ and let $\{(v_j,\eta_j)\}_{j\in\Nn_0}$ be a trajectory.
We consider three cases:

\smallskip

\noindent
\blob \; If $[0,k]$ contains no $\delta$-collinearity, by Proposition~\ref{expans} $L_{[0,k]}$ is a
$\sigma''$-expansion on $ v_0^\perp\cap V_{[0,k]}$.
But since any trajectory is generating on $[0,k]$,
we have $ v_0^\perp = v_0^\perp\cap V_{[0,k]}$, which proves that $L_{[0,k]}$ is a
$\sigma''$-expansion. Finally, because $L_{[k,2k]}$
is non contracting, $L_{[0,2k]}= L_{[k,2k]}\circ L_{[0,k]}$ is also a $\sigma''$-expansion. 

\noindent
\blob \; If $[0,k]$ contains a $\delta$-collinearity $[i,j]\subseteq [0,k]$,  we can assume it is minimal, in the sense that $[i,j]$ contains no proper subinterval which is itself a $\delta$-collinearity.
Consider first the case $j\geq i+1$.
By Proposition~\ref{expans}, $L_{[i,j]}$ is a
$\sigma''$-expansion on $ v_i^\perp\cap V_{[i,j]}$. Because
$L_{[i,2k]}=  L_{[j,2k]}\circ L_{[i,j]}$, and
$L_{[j,2k]}$ is non contracting, the map 
$L_{[i,2k]}$ is also a $\sigma''$-expansion on $ v_i^\perp\cap V_{[i,j]}$.
Remark that since $i\leq k$, 
the trajectory is generating on $[i,2k]$,
and hence  $ v_i^\perp = v_i^\perp\cap V_{[i,2k]}$.
Hence by Proposition~\ref{delta:rel:expans},
$L_{[i,2k]}$ is a relative $\sigma'$-expansion on
$ v_{i}^\perp$ w.r.t. $ v_i^\perp\cap V_{[i,j]}$.
Thus by Lemma~\ref{abstr:expans:lemma},
$L_{[i,2 k]}$ is a $\sigma_0$-expansion,
which implies so is $L_{[0,2k]}$.

\noindent
\blob \; Finally we consider the case  $[0,k]$ contains  $\delta$-collinearities, but the minimal ones have length zero, say $\{i\}\subset [0,k]$ is a $\delta$-collinearity.
In this case we have $\angle(v_i,\eta_i)<\delta$, and the proof is somehow simpler. By Lemma~\ref{angle:lemma}
\begin{align*}
\angle( V_{[i,j-1]}, N_{[i,j-1]} ) & =
\angle( \linspan{v_i}+N_{[i+1,j-1]}, \linspan{\eta_i}+N_{[i+1,j-1]} ) \\
&\leq \arcsin\left(\frac{\sin\delta}{\sin\varepsilon} \right)=:\hat\delta\;.
\end{align*}
On the other hand, because 
$v_j= \alpha_j\,\eta_j+\beta_j\,v_{j-1}$ with $\alpha_j\geq c$ and $c=\cos( \frac{\pi}{2}\lambda(f))$,  whenever $\eta_j\notin \{\eta_i,\ldots, \eta_{j-1}\}$ we have
\begin{align*}
\angle(v_j, V_{[i,j-1]} ) &\geq
\frac{c}{2}\angle(\eta_j,V_{[i,j-1]} ) \\
 &\geq
\frac{c}{2}\,\angle(\eta_j, N_{[i,j-1]}) -\frac{c\, \hat\delta}{2}  \\
 &\geq
\frac{c}{2}\,(\varepsilon -\hat\delta) \geq \frac{c\,\varepsilon}{4}\;,
\end{align*}
provided $\delta$ is small enough.
Thus, using Lemma~\ref{proj:sing:val} we get by induction that $L_{[i,i+k]}$ is a uniform expansion, and as before that $L_{[0,2k]}$ is also a uniform expansion.

\smallskip

Therefore,  $L_{[0,2k]}$
is a $\sigma_0$-expansion in all cases.
\end{proof}

\section{Proof of the Main Statements}
\label{sec:proof}
Mohammad

Let $P$ be a spanning polytope and $f$ a contracting reflection law.
Denote by $\Phi=\Phi_{f,P}\colon D\to D$ the billiard map for $P$ and  $f$.

\begin{proof}[Proof of Theorem~\ref{main theorem}]
Let $x=(p,v)\in D$ be any $k$-generating point. We can identify the tangent space $T_x M$ with $v^\perp\times v^\perp$ using the Jacobi coordinates $(J,J')$. From the proof of Proposition~\ref{prop:partialhyperbolic}, the subbundle $E^{cu}(x)$ in the coordinates $(J,J')$ is $\{(J,J')\in v^\perp\times v^\perp\colon J'=0\}$. Moreover, by Theorem~\ref{hyperb:main}, there exists $\sigma>1$ depending only on $P$ and $f$ such that
$$
\|D\Phi^{2k}(x)(J,0)\|= \| L_{[0,2k]}(J)\|\geq \sigma\|J\|,\quad\forall\,J\in v^\perp.
$$
This uniform minimum growth expansion on $E^{cu}$ proves the theorem.
\end{proof}

\begin{proof}[Proof of Theorem~\ref{coro NUH} ]
Assume that $(\Phi,\mu)$ is ergodic and $\int T\,d\mu<+\infty$. First note that, by Proposition~\ref{prop:partialhyperbolic},
$$
\limsup_{n\to\infty}\frac{1}{n}\log\|D\Phi^n(x)\vert_{E^{s}}\|=\log\lambda(f)<0
$$
for every $x\in D$.
Consider now the partition $\{A_n=T^{-1}\{n\}\}_{n\in\Nn}$ of $D$,
and define the measurable function $\tilde T:D\to\Nn$,
$\tilde T=n$ on $A_n':= \Phi(A_n)$. This function satisfies
$$ T\left( \Phi^{ -\tilde T(x) }(x) \right)= \tilde T(x)\quad \text{ for all } \; x\in D\,.  $$
Moreover $\int \tilde T\,d\mu 
= \int  T\,d\mu <+\infty$.
From Theorem~\ref{main theorem} we have
$$ \norm{ D\Phi^{-2\,\tilde T(x)}(x) \vert_{E^{cu}} }\leq 1/\sigma \quad \text{ for all } \; x\in D\,. $$
Define recursively the following sequence of backward iterates and stopping times
$$ \left\{ \begin{array}{ll}
x_0 &= x \\
t_0 &= 2 \,\tilde T(x_0)
\end{array}\right. \qquad
\left\{ \begin{array}{ll}
x_{j+1} &= \Phi^{-t_j}(x_j) \\
t_{j+1} &= 2 \,\tilde T(x_{j+1})
\end{array}\right. \;. $$
Let us write $\tau_n=\sum_{j=0}^{n-1} t_j $.
Since $t_j\geq 2\,d$ for all $j$, this sequence tends to $+\infty$, and we have
\begin{align*}
- \frac{1}{\tau_n}\, \log \norm{ D\Phi^{-\tau_n}(x)\vert_{E^{cu}} } &\geq
- \frac{1}{\sum_{j=0}^{n-1} t_j}	\, \sum_{j=0}^{n-1}
\log \norm{ D\Phi^{-t_j}(x_j)\vert_{E^{cu}} } \\
&\geq
- \frac{n}{\sum_{j=0}^{n-1} t_j}	\, 
\log \sigma^{-1} =
\frac{\log \sigma}{\frac{1}{n}\,\sum_{j=0}^{n-1} \tilde T(x_{j-1})}\;.
\end{align*}
Thus, by Birkhoff's ergodic theorem, for $\mu$-almost every $x\in D$,
$$ \limsup_{n\to+\infty}
-\frac{1}{n}	\, \log \norm{ D\Phi^{-n}(x)\vert_{E^{cu}} }\geq 
 \frac{\log \sigma}{ \int \tilde T\,d\mu } >0\;.$$
By Kingman's ergodic theorem,  the above $\limsup$ is actually a limit. Thus,
$$
\lim_{n\to\infty}\frac{1}{n}\log \norm{ D\Phi^{-n}(x)\vert_{E^{cu}} }<0
$$
for $\mu$-almost every $x\in D$.
This proves that $\mu$ is a  hyperbolic measure.
\end{proof}

\begin{proof}[Proof of Theorem~\ref{coro UH} ]
Assume that $\Lambda\subset D$ is $\Phi$-invariant. 
By Proposition~\ref{prop:partialhyperbolic}, $\Phi$ is uniformly partially hyperbolic on $\Lambda$. Moreover, it follows from Theorem~\ref{main theorem} that there exists a constant $C>0$ depending only on $P$ and $f$ such that
$$
\| D\Phi^{-n}(x)|_{E^{cu}(x)}\|\leq C \left(\frac1\sigma\right)^{\frac{n}{2k}}
$$
for every $x\in \Lambda$ that is $k$-generating. Since the escaping time function $T$ is bounded on $\Lambda$, every $x\in \Lambda$ is $\tau$-generating where $\tau:=\sup_{x\in\Lambda}T(x)$. So the expansion rate can be made uniform and equal to $\sigma^{1/\tau}>1$. This shows that $\Phi$ is uniformly hyperbolic on $\Lambda$.
\end{proof}

\begin{proof}[Proof of Corollary~\ref{coro 1}]
Suppose $P$ is in general position, in particular $P$ is a spanning polytope.
By Corollary~\ref{cor escaping 1} there exists a positive 
constant $\lambda_0=\lambda_0(P)$ such that the escaping time function  $T$ is bounded on $D$.
The claim follows by Theorem~\ref{coro UH}. 
\end{proof}

\begin{proof}[Proof of Corollary~\ref{coro 2}]
Suppose $P$ is an obtuse polytope in general position, in particular $P$ is a spanning polytope. By Corollary~\ref{cor escaping 2}
the escaping time function  $T$ is bounded on $D$.
The claim follows by Theorem~\ref{coro UH}.
\end{proof}

\section{Examples}
\label{sec:examples}

In this section we study in detail the contracting billiard on a family of $3$-dimensional simplexes, illustrating the applicability of our main theorems.

Let $\{e_1,\ldots, e_{d+1}\}$ be the canonical basis of $\Rr^{d+1}$.
Given $d\geq 2$, we denote by $\Delta^{d}_h$ the  
$d$-simplex in $\Rr^{d+1}$ defined as the convex hull of the vertexes 
$v_j=e_j$ for  $1\leq j\leq d$  and $v_{d+1}= \frac{1-h}{d}\sum_{j=1}^d e_j + h\, e_{d+1}$.
For any set of $d$ facets of $\Delta^{d}_h$ ($(d-1)$-dimensional faces), their normals are linearly independent. Therefore, $\Delta^{d}_h$ is in general position according to Definition~\ref{def generic polytopes}  and it is  spanning according to Definition~\ref{def spanning}.

\subsection{Near conservative billiards}
We firstly consider contracting reflection laws close to the specular one. It will be shown that the escaping time is uniformly bounded, by computing explicitly the barycentric angle of  $\Delta^{d}_h$.

The simplex $\Delta^{d}_h$ has $d+1$ barycentric angles,
one for each vertex. By symmetry all  barycentric angles at base vertexes $v_j$, with $1\leq j\leq d$, are the same.
Denote the barycentric  angle at the base vertexes by 
 $\phi_1=\phi_1(h)$ and the barycentric  angle at
$v_{d+1}$  by $\phi_2=\phi_2(h)$.
Define 
\begin{equation}
\label{lambda0}
\lambda_0(h):= 1-4\,\min\{ \phi_1(h),\phi_2(h)\} /\pi . 
\end{equation}

\begin{proposition}\label{example1}
For every $h>0$ and every contracting reflection law $f$ satisfying $\lambda_0(h) < \lambda(f)<1$ the billiard map $\Phi_{f,\Delta^{d}_h}$ is uniformly hyperbolic.
\end{proposition}

\begin{proof}
Notice that $\Delta^{d}_h$ is in general position and spanning.
Moreover, by Theorem \ref{th:zigzag}, if $2\phi_i>\pi/2-f(\pi/2)>\pi/2-\lambda(f)\pi/2$ for $i=1,2$  then the polyhedral cones have bounded escaping time. This is the case when  $\lambda(f)  >1-4\min\{\phi_1,\phi_2\}/\pi$. Thus, by Theorem~\ref{coro UH},
$\Phi_{f,\Delta^{d}_h}$ is uniformly hyperbolic.
\end{proof}

Figure \ref{fig:1} shows the graphs of the $\lambda_0(h)$ defined in~\eqref{lambda0} for $d=3$, $4$ and $5$.
The shaded  regions bounded between these graphs and $\lambda=1$ 
are called admissible regions.
This figure shows that the admissible regions decrease as the dimension increases. The bottom tips of these admissible regions correspond to the heights $h$ of the regular $d$-simplexes.

\begin{figure}[h]
\includegraphics[scale=.8]{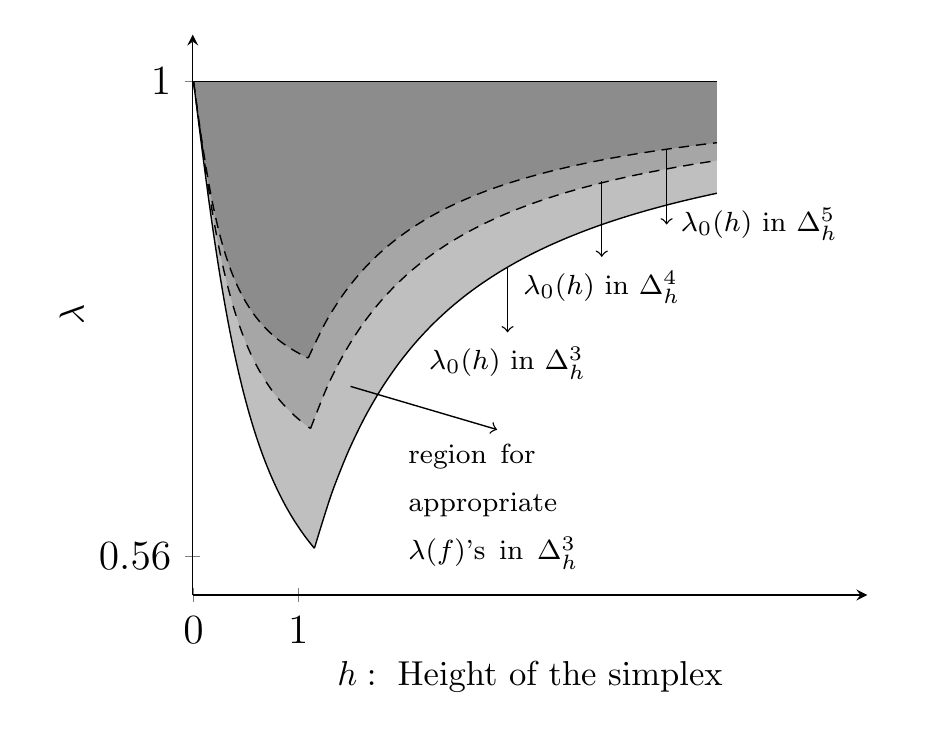}
\caption{Parameter regions with uniform bounded escaping time}
\label{fig:1}
\end{figure}

\subsection{Near slap billiards} 

Here we consider the situation when $\lambda(f)\approx 0$ for a given contracting reflection law $f$.
These reflection laws are called strongly contracting (see~\cite{MDDGP13}). In this context the dynamics may loose uniformity due to
unbounded escaping times.
To any strongly contracting billiard
we can associate a degenerate billiard map called the `slap map'
corresponding to $f=0$, where reflections are always orthogonal to the faces. When $h$ is  small enough the slap map has a trapping region,  called a  \emph{chamber}, away from acute wedges.
Hence the escaping time is bounded on the chamber. 
This concept generalizes the notion of chamber introduced in~\cite{MDDG14}.

For simplicity we will assume $d=3$.

\begin{proposition}
For any $h\in(0,1/2)$ there 
exists   $\lambda_0=\lambda_0(h)>0$ such that
for  every contracting reflection law $f$ satisfying $\lambda(f)<\lambda_0(h)$, 
the  billiard map $\Phi_{f,\Delta^{3}_h}$ is uniformly hyperbolic.
\end{proposition}

\begin{proof} 
Firstly, let us assume that $\lambda(f)=0$. This means that the billiard particle always reflects orthogonally to each face of the polytope. Since after the first iterate the angle is zero, we can reduce $\Phi_{f,\Delta_h^3}$ to a multi-valued map $\Phi_0:\Delta^3_h\to \Delta_h^3$ (skeleton points may have more than one image). Let $A_i$, $i=1,\ldots,4$ denote the vertexes of the simplex $\Delta_h^3$ with $A_4$ being the top vertex. The triangle $A_1A_2A_3$ is called the base of the simplex (see Figure \ref{trapping}). We show that there is a set $\Vscr$  on the base of the simplex which is invariant by $\Phi_0^2$. Let $C_0$ denote the center of $A_1 A_2 A_3$, i.e., the point mapped by $\Phi_0$ to the top vertex of the simplex. Then, the base triangle is partitioned into three triangles, namely $A_1A_2C_0$, $A_1A_3C_0$ and $A_2A_3C_0$. Since $\Phi_0(C_0)$ is the intersection of the three faces, it has three distinct images by $\Phi_0$. A simple calculation shows that when $h <h_0$ for some $h_0>0$, these images belong to the base of the simplex. Denote them by $C_1$, $C_2$ and $C_3$. The image of triangles $A_1A_2C_0, A_2A_3C_0$ and $A_3A_1C_0$ under $\Phi^2_0$ are respectively triangles $A_1A_2C_3, A_2A_3C_1$ and $A_3A_1C_2$. Therefore $\Phi^2_0$ maps the triangle $A_1A_2A_3$ to itself.

Now we construct an hexagon $\Hscr=M_1 M_2 M_3 M_4 M_5 M_6$ as follows (see Figure~\ref{trapping}): the point $M_1$ is the intersection of $A_1 C_2$ with the perpendicular to $A_1 C_0$ through $C_1$. Likewise, $M_2$ is the intersection of $A_2 C_1$ with the perpendicular to $A_2 C_0$ through $C_2$. The other $M_j$'s are similarly defined. The hexagon $\Hscr$ is the union of three pentagons whose images under $\Phi^2_0$ are in the hexagon $\Hscr$. On the right of Figure~\ref{trapping}, we can see the image of the pentagon $\Pscr=C_0 C_2 M_2 M_1 C_1$, $\Phi^2_0(\Pscr)=C_0'C_2'M_2'M_1' C_1'$. 

\begin{figure}[h]
\includegraphics[scale=.5]{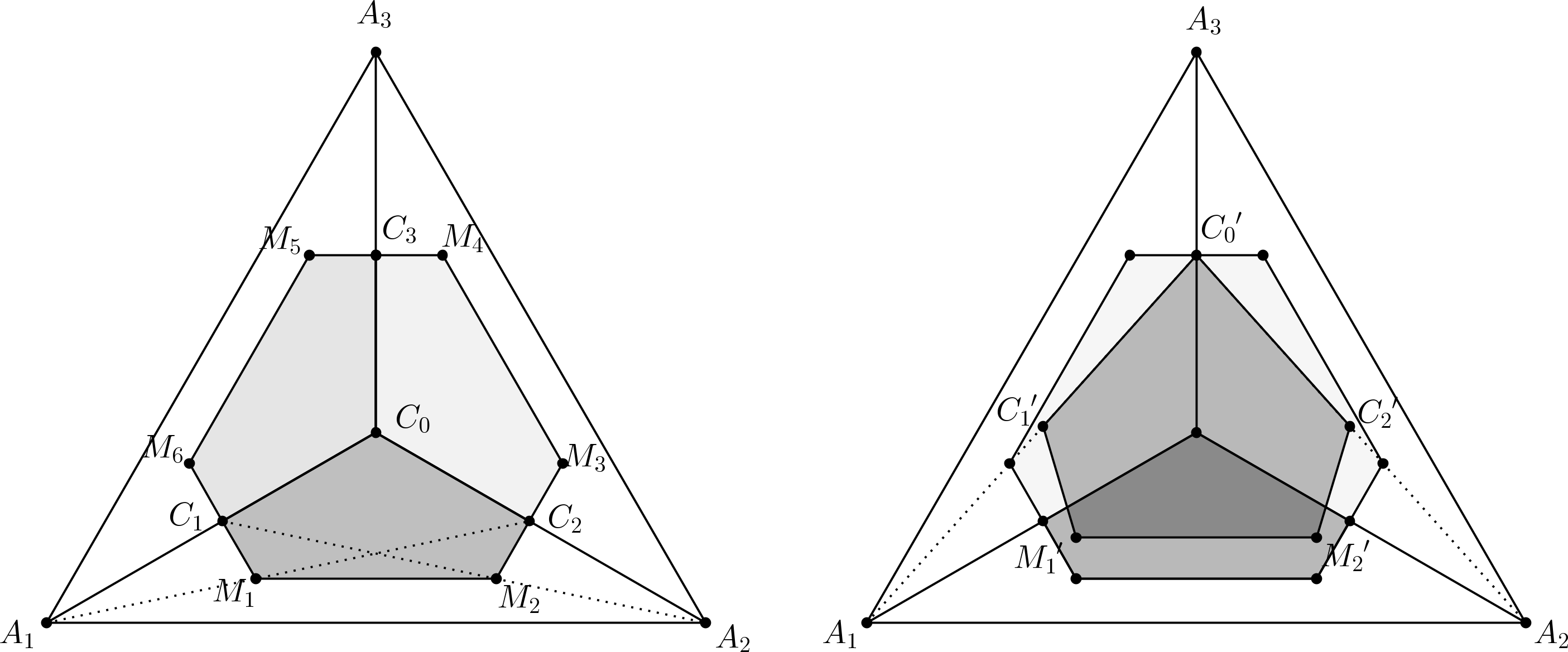}
\caption{}
\label{trapping}
\end{figure}

Moreover, the intersection of the pentagon  $\Phi^2_0(\Pscr)$ with the boundary of $\Hscr$ is just the point $C_0'=C_3$. Hence, for some small enough neighborhood $\Vscr$ of $\Hscr$ on the base triangle $A_1 A_2 A_3$ we have $\Phi^2_0(\overline{\Vscr})\subset  \Vscr$.

It is easy to see that every orbit of $\Phi_0$ eventually enters $\Vscr$. In fact, every orbit starting near the wedges of the simplex will escape by a zig-zag movement and enter $\Hscr\subset \Vscr$.  Since $\Vscr$ is away from the wedges, the escaping time $T(x)$ for $x\in\Vscr$ is uniformly bounded.

Denote by $\eta$ the inward normal to the base of the simplex and given $\lambda_0>0$ define
$$
\mathbb{S}^+_{\eta,\lambda_0}=\left\{v\in \mathbb{S}^+_{\eta}\colon \langle v,\eta\rangle>\cos\left(\lambda_0\frac\pi2\right)\right\}.
$$ 
Also define $\Lambda_{\lambda_0}:=\left(\Vscr\times\mathbb{S}^+_{\eta,\lambda_0}\right)\cap M^+$. Then, by continuity, there exists $\lambda_0=\lambda_0(h)>0$ such that for every contracting reflection law satisfying $\lambda(f)<\lambda_0$ we have
$$
\Phi_{f,\Delta_h^3}^2(\Lambda_{\lambda_0})\subset \Lambda_{\lambda_0}\quad\text{and}\quad D=\bigcap_{n\geq0}\Phi_{f,\Delta_h^3}^n(\Lambda_{\lambda_0}).
$$  
The previous equality follows from the fact that for each $x\in M^+$ there exists $n\geq0$ such that $\Phi_{f,\Delta_h^3}^n(x)\in \Lambda_{\lambda_0}$. Since the escaping time is bounded on $\Lambda_{\lambda_0}$, it is also bounded on $D$.
Therefore, the proposition follows from Theorem~\ref{coro UH}.

\end{proof}

\section*{Acknowledgements}

The authors were partially supported by Funda\c c\~ao para a Ci\^encia e a Tecnologia (FCT/MEC). 
PD was supported  under the project: UID/MAT/04561/2013.
JPG was supported through the FCT/MEC grant SFRH/BPD/78230/2011 and the project UID/Multi/00491/2013 financed by FCT/MEC through national funds and when applicable co-financed by FEDER, under the Partnership Agreement 2020. MS was supported by PNPD/CAPES. The authors wish to express their gratitude to Gianluigi Del Magno for stimulating conversations and also to the anonymous referee that helped us to significantly improve the presentation of the paper.


\bibliographystyle{plain}

\begin{thebibliography}{1}


\bibitem{arroyo09}
A.~ Arroyo, R.~ Markarian, and D.~P.~Sanders.
\newblock {Bifurcations of periodic and chaotic attractors in pinball billiards
  with focusing boundaries.}
\newblock {\em Nonlinearity}, 22(7):1499--1522, 2009.

\bibitem{arroyo12}
A.~ Arroyo, R.~ Markarian, and D.~P.~Sanders, Structure and evolution of strange attractors in non-elastic triangular billiards, Chaos {\bf 22}, 2012, 026107.




\bibitem{DK-book}
Pedro Duarte and Silvius Klein,  {L}yapunov exponents of linear cocycles; continuity via large
  deviations, Atlantis Studies in Dynamical Systems, vol.~3, Atlantis Press,
  2016.
  


\bibitem{MDDGP12}
G.~Del Magno, J.~Lopes Dias, P.~Duarte, J.~P.~Gaiv\~ao\ and D.~Pinheiro, Chaos in the square billiard with a modified reflection law, Chaos {\bf 22}, (2012), 026106.

\bibitem{MDDGP13}
G.~Del Magno, J.~Lopes Dias, P.~Duarte, J.~P.~Gaiv\~ao and D.~Pinheiro, SRB measures for polygonal billiards with contracting reflection laws, Comm. Math. Phys. {\bf 329} (2014), 687--723.

\bibitem{MDDG14}
G.~Del Magno, J.~Lopes Dias, P.~Duarte and J.~P.~Gaiv\~ao,
Ergodicity of polygonal slap maps,
Nonlinearity, {\bf 27}, 8, (2014), 1969--1983


\bibitem{MDDG15}
G.~Del Magno, J.~Lopes Dias, P.~Duarte and J.~P.~Gaiv\~ao, Hyperbolic polygonal billiards with finitely may ergodic SRB measures, to appear in  Ergodic Theory Dyn. Syst.  (2016).


   
   
   

\bibitem{markarian10}
R.~ Markarian, E.~J.~Pujals, and M.~ Sambarino.
\newblock Pinball billiards with dominated splitting.
\newblock {\em Ergodic Theory Dyn. Syst.}, 30(6):1757--1786, 2010.



\bibitem{Sinai78}
Ya.~G.~ Sinai,
\newblock Billiard trajectories in a polyhedral angle.
\newblock {\em Russian Math. Surveys}, 33:1, 219--220, 1978.



\bibitem{Stern}
S.~Sternberg, 
\newblock Lectures on differential geometry,
\newblock Prentice-Hall, Inc., Englewood Cliffs, N.J., 1964


 



\end{thebibliography}

\end{document}